\documentclass[a4paper]{article}
\usepackage[a4paper, margin = 1.1in, bottom = 1in]{geometry}

\usepackage{hyperref,authblk,amsmath,amssymb,amsthm,color,tikz,comment,graphics}
\usepackage[font={small},width=1.14\linewidth]{caption}
\usetikzlibrary{arrows}
\usetikzlibrary{arrows.meta}

\usepackage{mathrsfs,dsfont}
\newcommand{\C}{\mathds{C}}
\newcommand{\R}{\mathds{R}}

\newtheorem{nummer}{ }
\newtheorem{theorem}[nummer]{\bf Theorem}
\newtheorem{lemma}[nummer]{Lemma}
\newtheorem{corollary}[nummer]{Corollary}

\theoremstyle{definition}
\newtheorem{definition}[nummer]{Definition}

\definecolor{darkgreen}{rgb}{0.,0.6,0.}
\def\epsilon{\varepsilon}
\newcommand{\intersect}[1][.7ex]{%
  \tikz[baseline=-2*#1]{%
  \draw[thin] (0,0) arc[start angle=90, end angle=-90, radius=#1];
  \draw[thin] (1.5*#1,-2*#1) arc[start angle=270, end angle=90, radius=#1];
  }%
}

\newcommand{\touch}[1][.7ex]{%
  \tikz[baseline=-2*#1]{%
  \draw[thin] (0,0) arc[start angle=90, end angle=-90, radius=#1];
  \draw[thin] (2*#1,-2*#1) arc[start angle=270, end angle=90, radius=#1];
  }%
}

\begin{document}

{\LARGE\bf Closing Theorems for Circle Chains}\\[2mm]
{\small Norbert Hungerb\"uhler\\ Department of Mathematics, ETH Z\"urich, 8092 Z\"urich, Switzerland\\
norbert.hungerbuehler@math.ethz.ch\\
ORCHID 0000-0001-6191-0022}

\begin{center}
\parbox{.9\linewidth}{\small
{\bf Abstract.} We consider closed chains of circles $C_1,C_2,\ldots,C_n,C_{n+1}=C_1$ such that two neighbouring circles $C_i,C_{i+1}$ 
intersect or touch each other with $A_i$ being a common point. We formulate conditions such that
a polygon with vertices $X_i$ on $C_i$, and $A_i$ on the (extended) side $X_iX_{i+1}$, is closed for
every position of the starting point $X_1$ on $C_1$. Similar results apply to open chains of circles. It turns out that the intersection of 
the sides $X_iX_{i+1}$ and $X_jX_{j+1}$ of the polygon lies on a circle $C_{ij}$  through $A_i$ and $A_j$
with the property that $C_{ij}, C_{jk}$ and $C_{ki}$ pass through a common point.
The six circles theorem of Miquel and Steiner's quadrilateral Theorem appear as special cases
of the general results.}
\end{center}

{\bf Key words.} closing theorems, circle chains, Miquel's six circles theorem, Steiner's quadrilateral theorem

{\bf Mathematics Subject Classification.} 51M04, 51M15

\section{Introduction}
The treasure trove of geometry contains a wide spectrum of closing theorems. 
Among the best known are Steiner's closing theorem (Figure~\ref{fig-steiner} on the left,~\cite[\S\,6.5]{coxeter1}), which belongs to the M\"obius geometry, and 
Poncelet's porism (Figure~\ref{fig-steiner} on the right,~\cite{flatto,drago,hhp}), one of the deepest results of projective geometry. 
\begin{figure}[h!]
\begin{center}
\begin{tikzpicture}[line cap=round,line join=round,x=3.2,y=3.2]
\draw [line width=0.8pt,color=blue] (20.902639991454564,52.3627225046939) circle (10.942812444877323);
\draw [line width=0.8pt,color=blue] (24.307761937980693,60.89281514626856) circle (28.13792902779854);

\draw [line width=0.8pt,color=red] (23.119015762188283,37.18276352257279) circle (4.398096141512681);
\draw [line width=0.8pt,color=red] (14.850437583905874,38.692555746674806) circle (4.007191340205413);
\draw [line width=0.8pt,color=red] (8.015786233901109,43.8038632996496) circle (4.527322352865745);
\draw [line width=0.8pt,color=red] (3.693217436107717,53.75589474066403) circle (6.322909479431505);
\draw [line width=0.8pt,color=red] (8.577900222399228,69.5317793550592) circle (10.191891799706049);
\draw [line width=0.8pt,color=red] (31.5036155162772,74.01889780138099) circle (13.168815674325444);
\draw [line width=0.8pt,color=red] (41.37095539655238,53.544153574099894) circle (9.55957070670588);
\draw [line width=0.8pt,color=red] (32.953008088949915,40.4867744628995) circle (5.976097483854311);

\draw [line width=0.5pt,color=red,dash pattern=on 2pt off 2pt] (20.86654897841823,37.21178158450317) circle (4.2081714614018715);
\draw [line width=0.5pt,color=red,dash pattern=on 2pt off 2pt] (12.972905565294186,39.6384991925038) circle (4.050070765763476);
\draw [line width=0.5pt,color=red,dash pattern=on 2pt off 2pt] (6.5615107897159675,45.787854613432025) circle (4.83365340044131);
\draw [line width=0.5pt,color=red,dash pattern=on 2pt off 2pt] (3.5655555946234427,57.331228868387484) circle (7.092168865590082);
\draw [line width=0.5pt,color=red,dash pattern=on 2pt off 2pt] (12.971963886823229,73.23175906834332) circle (11.382337287465063);
\draw [line width=0.5pt,color=red,dash pattern=on 2pt off 2pt] (36.84321898265364,69.85879063762884) circle (12.726037520572213);
\draw [line width=0.5pt,color=red,dash pattern=on 2pt off 2pt] (39.970938014545965,48.93390319939474) circle (8.431314408372696);
\draw [line width=0.5pt,color=red,dash pattern=on 2pt off 2pt] (30.329878680537533,38.988031008643425) circle (5.420414679240572);
\end{tikzpicture}
\begin{tikzpicture}[line cap=round,line join=round,x=30,y=30]
\draw [rotate around={-6.3401917459099115:(0.99,0.87)},line width=0.8pt,color=darkgreen] (0.99,0.87) ellipse (2.4826515605076316 and 1.9477573696153683);
\draw [rotate around={148.49233716521107:(1.2382399646487747,0.3486206623337557)},line width=0.8pt,color=blue] (1.2382399646487747,0.3486206623337557) ellipse (3.189458639461491 and 2.657041844942135);
\draw [line width=0.5pt,dash pattern=on 2pt off 2pt,color=red] (-1.1730928663927847,-0.9952286406115576)-- (4.0226490672376265,-1.2044865241168552);
\draw [line width=0.5pt,dash pattern=on 2pt off 2pt,color=red] (4.0226490672376265,-1.2044865241168552)-- (3.139853233240123,2.2372264069124297);
\draw [line width=0.5pt,dash pattern=on 2pt off 2pt,color=red] (3.139853233240123,2.2372264069124297)-- (0.4054941135684272,3.1428397457375215);
\draw [line width=0.5pt,dash pattern=on 2pt off 2pt,color=red] (0.4054941135684272,3.1428397457375215)-- (-1.6354056674394561,1.7141041496119653);
\draw [line width=0.5pt,dash pattern=on 2pt off 2pt,color=red] (-1.6354056674394561,1.7141041496119653)-- (-1.1730928663927847,-0.9952286406115576);
\draw [line width=0.8pt,color=red] (1.8310772816612164,-2.4623111350583935)-- (3.9617494106224274,1.1980999166036632);
\draw [line width=0.8pt,color=red] (3.9617494106224274,1.1980999166036632)-- (1.6773183684541484,3.0300026954767403);
\draw [line width=0.8pt,color=red] (1.6773183684541484,3.0300026954767403)-- (-1.0542931411641128,2.522885135232743);
\draw [line width=0.8pt,color=red] (-1.0542931411641128,2.522885135232743)-- (-1.7760094450970916,0.3547276739543823);
\draw [line width=0.8pt,color=red] (-1.7760094450970916,0.3547276739543823)-- (1.8310772816612164,-2.4623111350583935);
\begin{scriptsize}
\draw [fill=red] (0.4054941135684272,3.1428397457375215) circle (1.5pt);
\draw [fill=red] (3.139853233240123,2.2372264069124297) circle (1.5pt);
\draw [fill=red] (4.0226490672376265,-1.2044865241168552) circle (1.5pt);
\draw [fill=red] (-1.1730928663927847,-0.9952286406115576) circle (1.5pt);
\draw [fill=red] (-1.6354056674394561,1.7141041496119653) circle (1.5pt);
\draw [fill=red] (1.8310772816612164,-2.4623111350583935) circle (1.5pt);
\draw [fill=red] (-1.7760094450970916,0.3547276739543823) circle (1.5pt);
\draw [fill=red] (3.9617494106224274,1.1980999166036632) circle (1.5pt);
\draw [fill=red] (1.6773183684541484,3.0300026954767403) circle (1.5pt);
\draw [fill=red] (-1.0542931411641128,2.522885135232743) circle (1.5pt);
\end{scriptsize}
\end{tikzpicture}
\caption{Steiner's closing theorem on the left: 
Two blue support circles are given, into which a chain of circles is fitted so that neighbouring circles touch each other. 
If the chain closes in a certain position (red), it closes in every position (dashed). Poncelet's porism on the right:
A chain of tangents to a conic (green) with vertices on a second conic (blue) is drawn.
If the chain closes in a certain position (red), it closes in every position (dashed).}\label{fig-steiner}
\end{center}
\end{figure}
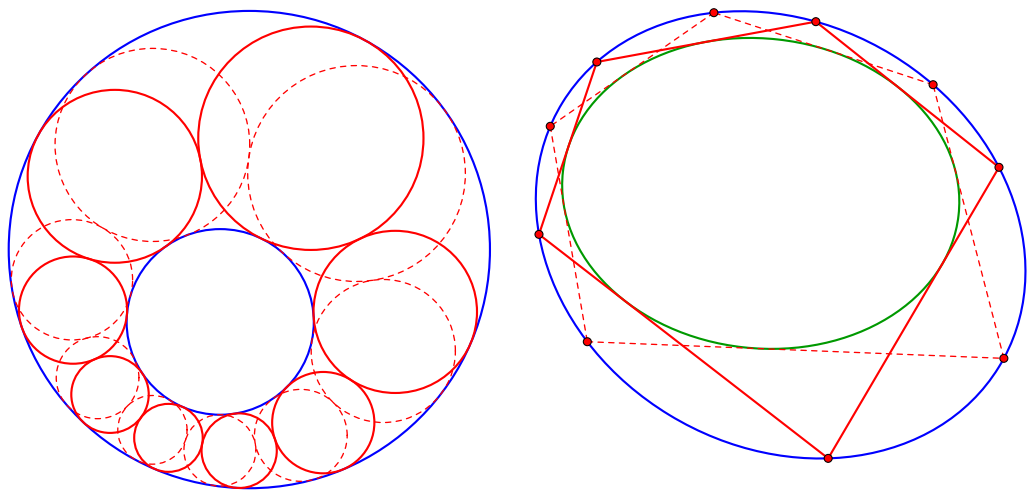

Other examples are the classical theorems of Pappus and Desargues, which have been known for a long time and are fundamental to the axiomatics of geometry. 
These theorems can be formulated in the same style as the theorems of Steiner and Poncelet. They then show their closing character much more clearly than usual.
Figure~\ref{fig-pappus} illustrates this using the example of the Pappus hexagon theorem. In this formulation, the theorem has very universal generalizations 
(see~\cite{nh-pappus,bh}).
Another well-known closing result is the Butterfly theorem in Figure~\ref{fig-pappus} on the right.  This theorem is also only a special case of a much more general closing result (see~\cite{hhs}).
\begin{figure}[ht!]
\begin{center}
\begin{tikzpicture}[line cap=round,line join=round,x=28,y=28]
\clip(1.0853464224195568,-2.9) rectangle (7.5472620504372285,3.1727244978543965);
\draw [line width=.8pt,domain=1.0853464224195568:7.5472620504372285] plot(\x,{(--3.3344--2.88*\x)/8.24});
\draw [line width=.8pt,domain=1.0853464224195568:7.5472620504372285] plot(\x,{(-1.218-4.62*\x)/12.6});
\draw [color=red,line width=.8pt] (2.186452981664151,1.1688573528146546)-- (3.549020753602968,-1.3979742763210883);
\draw [color=red,line width=.8pt] (3.549020753602968,-1.3979742763210883)-- (6.899196493869701,2.816029842517565);
\draw [color=red,line width=.8pt] (6.899196493869701,2.816029842517565)-- (2.2398931691857444,-0.9179608287014398);
\draw [color=red,line width=.8pt] (2.2398931691857444,-0.9179608287014398)-- (4.384641577686875,1.9371562795798785);
\draw [color=red,line width=.8pt] (4.384641577686875,1.9371562795798785)-- (5.303396103508056,-2.0412452379529538);
\draw [color=red,line width=.8pt] (5.303396103508056,-2.0412452379529538)-- (2.186452981664151,1.1688573528146546);
\draw [line width=0.8pt,domain=1.0853464224195568:7.5472620504372285] plot(\x,{(-0.23777535955809395--0.06676024673877502*\x)/0.49170980763435024});
\draw [color=red,line width=0.5pt,dash pattern=on 2pt off 2pt] (1.936548690442898,1.081512163649945)-- (4.008735027470491,-1.5665361767391803);
\draw [color=red,line width=0.5pt,dash pattern=on 2pt off 2pt] (4.008735027470491,-1.5665361767391803)-- (5.821691588092215,2.4394261861293174);
\draw [color=red,line width=0.5pt,dash pattern=on 2pt off 2pt] (5.821691588092215,2.4394261861293174)-- (2.399969998974154,-0.9766556662905234);
\draw [color=red,line width=0.5pt,dash pattern=on 2pt off 2pt] (2.399969998974154,-0.9766556662905234)-- (3.80911000192173,1.735999612322158);
\draw [color=red,line width=0.5pt,dash pattern=on 2pt off 2pt] (3.80911000192173,1.735999612322158)-- (6.440646425073458,-2.458237022526935);
\draw [color=red,line width=0.5pt,dash pattern=on 2pt off 2pt] (6.440646425073458,-2.458237022526935)-- (1.936548690442898,1.081512163649945);
\begin{small}
\draw[color=black] (1.3,1.06) node[yshift=1pt] {$\ell$};
\draw [fill=white] (2.8576618935811826,-0.09557902998259409) circle (1.5pt);

\draw [fill=white] (3.349371701215533,-0.028818783243819066) circle (1.5pt);

\draw [fill=white] (4.793379065273576,0.1672364640186292) circle (1.5pt);

\draw [fill=red] (2.186452981664151,1.1688573528146546) circle (1.5pt) node[above,yshift=-.5pt,red] {$X$};
\draw [fill=red] (3.549020753602968,-1.3979742763210883) circle (1.5pt);
\draw [fill=red] (6.899196493869701,2.816029842517565) circle (1.5pt);
\draw [fill=red] (2.2398931691857444,-0.9179608287014398) circle (1.5pt);
\draw [fill=red] (4.384641577686875,1.9371562795798785) circle (1.5pt);
\draw [fill=red] (5.303396103508056,-2.0412452379529538) circle (1.5pt);
\draw [fill=red] (1.936548690442898,1.081512163649945) circle (1.5pt);
\draw [fill=red] (4.008735027470491,-1.5665361767391803) circle (1.5pt);
\draw [fill=red] (5.821691588092215,2.4394261861293174) circle (1.5pt);
\draw [fill=red] (2.399969998974154,-0.9766556662905234) circle (1.5pt);
\draw [fill=red] (3.80911000192173,1.735999612322158) circle (1.5pt);
\draw [fill=red] (6.440646425073458,-2.458237022526935) circle (1.5pt);
\end{small}
\end{tikzpicture}\qquad
\begin{tikzpicture}[line cap=round,line join=round,x=36,y=36]
\clip(1.1047515744556455,-1.5621325989512826) rectangle (6.848676577138018,2.794324033150664);
\draw [rotate around={15.899522588456362:(4.044632107923147,0.6694598851989352)},line width=1pt] (4.044632107923147,0.6694598851989352) ellipse (2.4967319523547995 and 1.8608477180564371);
\draw [line width=0.8pt,domain=1.1047515744556455:6.848676577138018] plot(\x,{(-0.6526122654160171--0.7300696378479739*\x)/2.3409550964314114});
\draw [line width=0.8pt,color=red] (2.6323928159033807,2.0464485997043282)-- (3.7431921522022265,-1.2454222772792263);
\draw [line width=0.8pt,color=red] (3.7431921522022265,-1.2454222772792263)-- (5.484682990104053,2.376715431702518);
\draw [line width=0.8pt,color=red] (5.484682990104053,2.376715431702518)-- (5.312737522267382,-0.7974733948639078);
\draw [line width=0.8pt,color=red] (5.312737522267382,-0.7974733948639078)-- (2.6323928159033807,2.0464485997043282);
\draw [line width=0.5pt,dash pattern=on 2pt off 2pt,color=red] (2.246758947240675,1.740295055267992)-- (4.540898735427419,-1.1240175573683284);
\draw [line width=0.5pt,dash pattern=on 2pt off 2pt,color=red] (4.540898735427419,-1.1240175573683284)-- (5.160156689581066,2.4905762094337804);
\draw [line width=0.5pt,dash pattern=on 2pt off 2pt,color=red] (5.160156689581066,2.4905762094337804)-- (5.88306514829348,-0.3619796325604876);
\draw [line width=0.5pt,dash pattern=on 2pt off 2pt,color=red] (5.88306514829348,-0.3619796325604876)-- (2.246758947240675,1.740295055267992);
\begin{small}
\draw[color=black] (1.7,1.3) node {$C$};
\draw [fill=red] (2.6323928159033807,2.0464485997043282) circle (1.5pt);
\draw[fill=red] (2.58,2.25) node {$X$};
\draw [fill=red] (3.7431921522022265,-1.2454222772792263) circle (1.5pt);
\draw [fill=red] (5.484682990104053,2.376715431702518) circle (1.5pt);
\draw [fill=red] (5.312737522267382,-0.7974733948639078) circle (1.5pt);
\draw [fill=white] (3.091657835869947,0.685409666171449) circle (1.5pt);
\draw [fill=white] (5.4326129323013586,1.415479304019423) circle (1.5pt);
\draw[color=black] (1.25,0.25) node {$\ell$};
\draw [fill=red] (2.246758947240675,1.740295055267992) circle (1.5pt);
\draw [fill=red] (4.540898735427419,-1.1240175573683284) circle (1.5pt);
\draw [fill=white] (4.950193836292086,1.2650280906444482) circle (1.5pt);
\draw [fill=red] (5.160156689581066,2.4905762094337804) circle (1.5pt);
\draw [fill=red] (5.88306514829348,-0.3619796325604876) circle (1.5pt);
\draw [fill=white] (3.728082747110193,0.8838904081450525) circle (1.5pt);

\draw [fill=white] (1.5979801937539222,0.2195789900924332) circle (1.5pt);
\draw [fill=white] (6.3019513829672045,1.6865983907598925) circle (1.5pt);
\end{small}
\end{tikzpicture}

\caption{On the left: The hexagon theorem of Pappus formulated as closing theorem. If
the Pappus hexagon (red)  closes for one position  of the starting point $X$ on the line $\ell$, it closes 
for any other starting point (dashed). On the right: 
The Butterfly theorem. If the quadrangle (red) with starting point $X$ on a conic $C$ and passing through the four points
on a line $\ell$ is closed, then it closes for any other starting point (dashed).
 }\label{fig-pappus}
\end{center}
\end{figure}
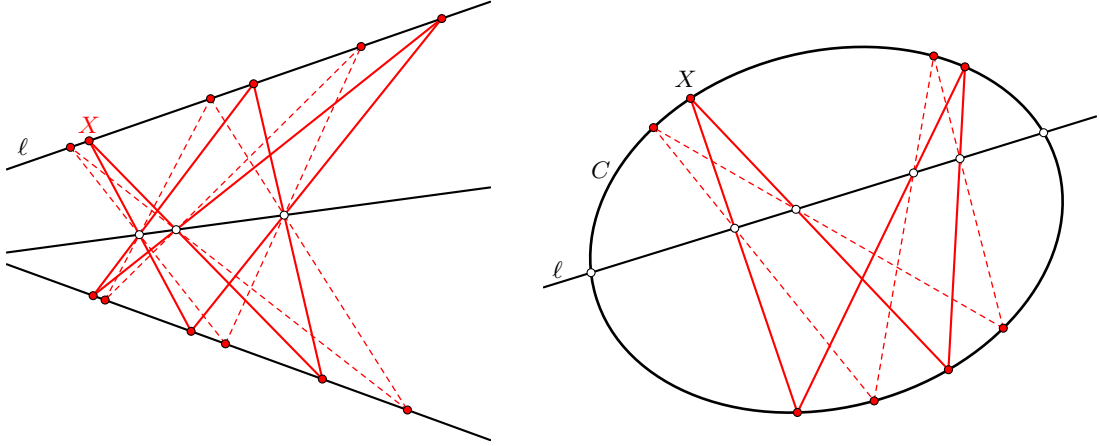
Other examples are the closing theorem of Emch~\cite{emch}, and the zig-zag theorem~\cite{bottema}.
In this article we investigate a family of new closing theorems for circle chains.
\section{A closing theorem for circle chains}
We will use the notation $C_1\mathbin{\intersect} C_2$ for two circles $C_1$ and $C_2$ that intersect either in two points or that touch each other.
To describe and prove the closing theorem, we will use the following map. 
\begin{definition}
Let  $C_1\mathbin{\intersect} C_2$ be two intersecting or touching circles and $A$ 
a common point of $C_1$ and $C_2$. Then $\varphi_A: C_1\to C_2, X\mapsto \varphi_A(X)$, is defined as follows: 
If $X\neq A$, then the points $X,\varphi_A(X)$, and $A$ are collinear. If $X=A$, then the line through the points $A$ and $\varphi_A(X)$ is the tangent to $C_1$ in $A$ (see Figure~\ref{fig-varphi}). The point $A$ will be called {\em pivot\/} of the map $\varphi_A$.
\end{definition}
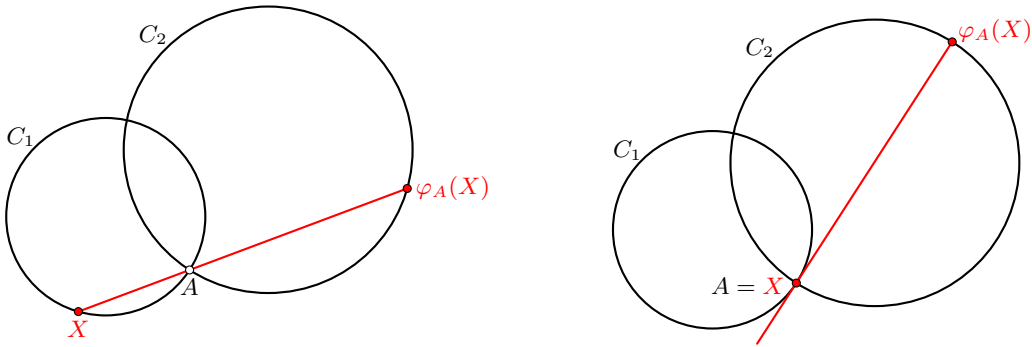
\begin{figure}[h!]
\begin{center}
\begin{tikzpicture}[line cap=round,line join=round,x=33,y=33]
\draw [line width=0.8pt] (1.5636224153067082,-0.5620894630386035) circle (1.128717693146846);
\draw [line width=0.8pt] (3.404928522943406,0.20296025774002305) circle (1.637740817958632);
\draw [line width=0.8pt,red] (1.25304140955378,-1.6472360290324497)-- (4.981804998062763,-0.23936944850065833);
\begin{small}
\draw [fill=white] (2.513894091177846,-1.1711767908166462) circle (1.5pt) node[below] {$A$};
\draw[color=black] (.6,0.3481389298034233) node {$C_1$};
\draw[color=black] (2.1,1.5177708734553814) node {$C_2$};
\draw [fill=red] (1.25304140955378,-1.6472360290324497) circle (1.5pt)  node[below,yshift=-0pt,xshift=0pt,red] {$X$};
\draw [fill=red] (4.981804998062763,-0.23936944850065833) circle (1.5pt) node[right,red] {$\varphi_A(X)$};
\end{small}
\end{tikzpicture}\qquad\qquad
\begin{tikzpicture}[line cap=round,line join=round,x=33,y=33]
\draw [line width=0.8pt] (1.5636224153067082,-0.5620894630386035) circle (1.128717693146846);
\draw [line width=0.8pt] (3.404928522943406,0.20296025774002305) circle (1.637740817958632);
\draw [line width=0.8pt,red] (4.281410085476945,1.586425131651382)-- (2.069754987026554,-1.8641033902936175);
\begin{small}
\draw [fill=red] (2.513894091177846,-1.1711767908166462) circle (1.5pt) node[left,yshift=-1pt,xshift=-1pt] {$A=\textcolor{red}X$};
\draw[color=black] (.6,0.3481389298034233) node {$C_1$};
\draw[color=black] (2.1,1.5177708734553814) node {$C_2$};
\draw [fill=red] (4.281410085476945,1.586425131651382) circle (1.5pt)  node[right,yshift=4pt,xshift=-1pt,red] {$\varphi_A(X)$};
\end{small}
\end{tikzpicture}
\caption{The map $\varphi_A:C_1\to C_2, X\mapsto \varphi_A(X)$.}\label{fig-varphi}
\end{center}
\end{figure}
Using such maps we can state the first theorem for a closed chain of $n$ circles as follows. Notice that  indices will be read cyclically throughout.
\begin{theorem}\label{thm-1}
Let $C_1\mathbin{\intersect} C_2 \mathbin{\intersect} C_3\mathbin{\intersect} \cdots \mathbin{\intersect}C_n\mathbin{\intersect}C_1$ be a closed
chain of circles. For $i=1,\ldots, n$ let $A_i$ be a common point of $C_i$ and $C_{i+1}$. 
Let $\varphi:=\varphi_{A_n}\circ\cdots\circ\varphi_{A_2}\circ\varphi_{A_1}$. Then the following holds: If
$\varphi(X)=X$ for one point $X\in C_1$, then $\varphi$ is the identity map on $C_1$.
\end{theorem}
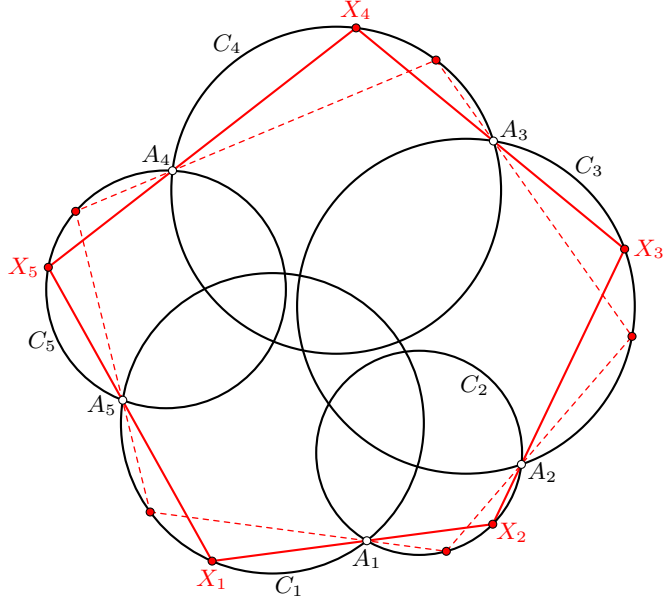
\begin{figure}[h!]
\begin{center}
\definecolor{red}{rgb}{1.,0.,0.}
\begin{tikzpicture}[line cap=round,line join=round,x=33,y=33]
\draw[line width=0.8pt,color=red] (1.6250382144525162,-3.7111624021193066) -- (4.807637352660276,-3.289299430790333) -- (6.302898026252813,-0.14352042511744068) -- (3.25755278530266,2.38498455825086) -- (-0.23220385792631373,-0.35216546773430646) -- cycle;
\draw[line width=0.5pt,dash pattern=on 2pt off 2pt,color=red] (0.9229338037076547,-3.149948458016592) -- (4.28119136297191,-3.599207981893942) -- (6.38673291962796,-1.1420375650823862) -- (4.16424290358626,2.0170457096106893) -- (0.07812198730851044,0.28779839471175905) -- cycle;
\draw [line width=0.8pt] (3.9729195054241004,-2.47344740849099) circle (1.1672053832914422);
\draw [line width=0.8pt] (4.503474201513647,-0.7973768912990093) circle (1.914537641156835);
\draw [line width=0.8pt] (3.0323907259926304,0.5290098489248604) circle (1.8695828611244967);
\draw [line width=0.8pt] (1.1031009220306414,-0.6044479109028101) circle (1.3589279916562766);
\draw [line width=0.8pt] (2.308907049506885,-2.135821692797642) circle (1.7173744303455507);
\begin{small}
\draw[color=black] (4.6,-1.7) node {$C_2$};
\draw[color=black] (5.9,.8) node {$C_3$};
\draw[color=black] (1.7964394453294776,2.2) node {$C_4$};
\draw[color=black] (-.3,-1.2) node {$C_5$};
\draw [fill=white] (0.611990829875993,-1.8715296619713093) circle (1.5pt)  node[left,yshift=-2pt,xshift=1pt] {$A_5$};
\draw[color=black] (2.5,-4) node {$C_1$};
\draw [fill=red] (1.6250382144525162,-3.7111624021193066) circle (1.5pt) node[below,red] {$X_1$};
\draw [fill=white] (3.3795802546498077,-3.4785926419150575) circle (1.5pt) node[below] {$A_1$};
\draw [fill=red] (4.807637352660276,-3.289299430790333) circle (1.5pt) node[right,yshift=-3pt,xshift=-2pt,red] {$X_2$};
\draw [fill=white] (5.132620830031038,-2.6055884107932252) circle (1.5pt)  node[right,xshift=-1pt,yshift=-3pt] {$A_{2}$};
\draw [fill=red] (6.302898026252813,-0.14352042511744068) circle (1.5pt) node[right,red] {$X_3$};
\draw [fill=white] (4.815321087840784,1.0915926423077253) circle (1.5pt) node[right,yshift=4pt,xshift=-1pt] {$A_3$};
\draw [fill=red] (3.25755278530266,2.38498455825086) circle (1.5pt)  node[above,red] {$X_4$};
\draw [fill=white] (1.1762153658246568,0.7525117659541158) circle (1.5pt) node[above,xshift=-6pt,yshift=-1pt] {$A_4$};
\draw [fill=red] (-0.23220385792631373,-0.35216546773430646) circle (1.5pt) node[left,red] {$X_5$};
\draw [fill=red] (0.9229338037076547,-3.149948458016592) circle (1.5pt);
\draw [fill=red] (4.28119136297191,-3.599207981893942) circle (1.5pt);
\draw [fill=red] (6.38673291962796,-1.1420375650823862) circle (1.5pt);
\draw [fill=red] (4.16424290358626,2.0170457096106893) circle (1.5pt);
\draw [fill=red] (0.07812198730851044,0.28779839471175905) circle (1.5pt);
\end{small}
\end{tikzpicture}
\caption{Theorem~\ref{thm-1} for five circles. If the polygon with starting point $X_1$ on $C_1$ closes, it closes for every starting point on $C_1$ (dashed).}\label{fig-thm1}
\end{center}
\end{figure}

The statement of Theorem~\ref{thm-1} can be formulated  in a geometric way visually as follows:
Let $X_1=X\in C_1$ and iteratively $X_{i+1}:=\varphi_{A_i}(X_{i}) \in C_{i+1}$, for $i=1,2,\ldots n$. 
Then we have: If $X=X_{n+1}$, then the corresponding polygon $X_1X_2\ldots X_nX_1$ through the
pivots $A_1,A_2,\ldots,A_n$ closes for every starting point $X_1$ on $C_1$ (see Figure~\ref{fig-thm1}).
\begin{proof}[Proof of Theorem~\ref{thm-1}]
Let $M_i$ denote the center of the circle $C_i$ (see Figure~\ref{fig-proof1}).
Take two chains $X_{i+1}:=\varphi_{A_i}(X_{i})$, and $X'_{i+1}:=\varphi_{A_i}(X'_{i})$ for $i=1,2,\ldots n$, for
two initial points $X_1,X_1'$ on $C_1$. 
Then we have that for all $i$ the angles $\sphericalangle X_iM_iX_i'$ have the same size,
namely $2\sphericalangle X_1A_1X_1'=:2\epsilon$. Then, from $\sphericalangle X_{n+1}M_1X_{n+1}' = 2\epsilon$
and $X_1=X_{n+1}$ it follows that $X'_1=X'_{n+1}$.
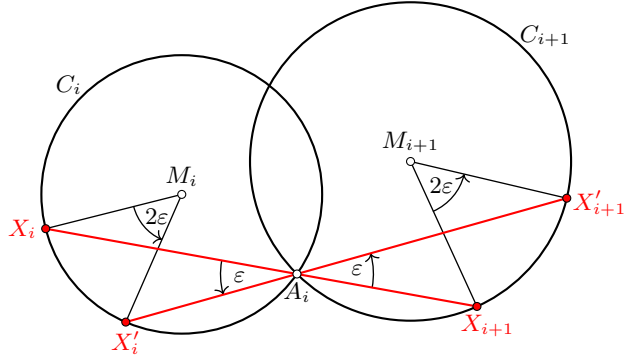
\begin{figure}[h!]
\begin{center}
\begin{tikzpicture}[line cap=round,line join=round,x=22,y=22]
\draw [shift={(3.77719617,-1.42461420)},line width=0.5pt,->] (169.6968672:1.3) arc (169.696867:195.8161425:1.3) ;
\draw [shift={(3.77719617,-1.42461420)},line width=0.5pt,->] (-10.3031327:1.3) arc (-10.3031327:15.816142569:1.3);
\draw [shift={(1.8178675163474398,-0.05970489092139551)},line width=0.5pt,->] (-165.7443896227277:0.8435598240297386) arc (-165.7443896227277:-113.50583898881722:0.8435598240297386);
\draw [shift={(5.70928032898988,0.5048937263753355)},line width=0.5pt,->] (-65.56804113355673:0.9372886933663762) arc (-65.56804113355673:-13.329490499646207:0.9372886933663762);
\draw [line width=0.8pt] (1.8178675163474398,-0.05970489092139551) circle (2.387874832082064);
\draw [line width=0.8pt] (5.70928032898988,0.5048937263753355) circle (2.730558558734794);
\draw [line width=0.5pt] (1.8178675163474398,-0.05970489092139551)-- (-0.4964770128468985,-0.6477147665860672);
\draw [line width=0.5pt] (1.8178675163474398,-0.05970489092139551)-- (0.865481491177273,-2.2494325161642053);
\draw [line width=0.5pt] (5.70928032898988,0.5048937263753355)-- (6.838673026079952,-1.9811517636427864);
\draw [line width=0.5pt] (5.70928032898988,0.5048937263753355)-- (8.366278557130874,-0.12463820953853944);
\draw [line width=0.8pt,red] (-0.4964770128468985,-0.6477147665860672)-- (6.838673026079952,-1.9811517636427864);
\draw [line width=0.8pt,red] (0.865481491177273,-2.2494325161642053)-- (8.366278557130874,-0.12463820953853944);
\begin{small}
\draw [fill=white] (1.8178675163474398,-0.05970489092139551) circle (1.5pt) node[above] {$M_i$};
\draw[color=black] (-0.1,1.8148724958113625) node {$C_i$};
\draw [fill=white] (5.70928032898988,0.5048937263753355) circle (1.5pt) node[above] {$M_{i+1}$};
\draw[color=black] (8,2.6584323198411033) node {$C_{i+1}$};
\draw [fill=red] (-0.4964770128468985,-0.6477147665860672) circle (1.5pt) node[left,red] {$X_i$};
\draw [fill=white] (3.7771961720268803,-1.4246142039251357) circle (1.5pt) node[below] {$A_i$};
\draw [fill=red] (0.865481491177273,-2.2494325161642053) circle (1.5pt) node[below,red] {$X_i'$};
\draw [fill=red] (6.838673026079952,-1.9811517636427864) circle (1.5pt) node[red,below,xshift=4pt] {$X_{i+1}$};
\draw [fill=red] (8.366278557130874,-0.12463820953853944) circle (1.5pt) node[red,right,yshift=-2pt,xshift=-1pt] {$X_{i+1}'$};
\draw[color=black] (2.8,-1.47) node {$\epsilon$};
\draw[color=black] (4.8,-1.35) node {$\epsilon$};
\draw[color=black] (1.5,-0.4) node[xshift=-2.3pt,yshift=-.6pt]  {2$\epsilon$};
\draw[color=black] (6.3,.1) node[xshift=-1.4pt] {2$\epsilon$};
\end{small}
\end{tikzpicture}
\caption{Proof of Theorem~\ref{thm-1}.}\label{fig-proof1}
\end{center}
\end{figure}
\end{proof}
The proof of Theorem~\ref{thm-1} was straightforward. The next result is more surprising.
\begin{theorem}\label{thm-2}
Let $C_1\mathbin{\intersect} C_2 \mathbin{\intersect} C_3\mathbin{\intersect} \cdots \mathbin{\intersect}C_n\mathbin{\intersect}C_1$ be a closed
chain of circles. Let $A_i,B_i$ be the intersection points of $C_i$ and $C_{i+1}$, where $A_i=B_i$ if $C_i$ touches $C_{i+1}$. 
Let $\varphi_A:=\varphi_{A_n}\circ\cdots\circ\varphi_{A_2}\circ\varphi_{A_1}$ and
$\varphi_B:=\varphi_{B_n}\circ\cdots\circ\varphi_{B_2}\circ\varphi_{B_1}$. Then the following holds: If
$\varphi_A(X)=X$ holds for one point $X\in C_1$, then $\varphi_B$ is the identity map on $C_1$.
\end{theorem}
Geometrically speaking, this means, that if the polygon  through the pivots $A_1, A_2,\ldots,A_n$ closes for one starting point $X$ on $C_1$,
then the polygon through the pivots $B_1, B_2,\ldots,B_n$ closes for each starting point on $C_1$ (see Figure~\ref{fig-thm2}).
\begin{figure}[h!]
\begin{center}
\begin{tikzpicture}[line cap=round,line join=round,x=44,y=44]
\draw [line width=0.8pt] (7.866527053974667,0.5199222482859711) circle (1.4576309783944001);
\draw [line width=0.8pt] (6.670800905123423,1.9230295974276448) circle (0.7804399469012693);
\draw [line width=0.8pt] (5.145492517869796,2.0388988601442075) circle (0.9223212140166329);
\draw [line width=0.8pt] (4.265830209481802,0.9540613325748167) circle (0.7478299812418407);
\draw [line width=0.8pt] (4.424352257628341,-0.6917748482134215) circle (0.9834255096672129);
\draw [line width=0.8pt] (6.270700754871843,-1.2715808475684112) circle (1.3596929026711528);

\draw [line width=.8pt,color=red] (5.9196990645862995,0.04202586792561247)-- (6.77798328814177,-0.44949022277130357);
\draw [line width=.8pt,color=red] (6.77798328814177,-0.44949022277130357)-- (7.337468452009789,1.5172699485617462);
\draw [line width=.8pt,color=red] (7.337468452009789,1.5172699485617462)-- (5.0432206312834635,2.9555323108971203);
\draw [line width=.8pt,color=red] (5.0432206312834635,2.9555323108971203)-- (3.614600171992503,0.5864350017498108);
\draw [line width=.8pt,color=red] (3.614600171992503,0.5864350017498108)-- (5.403456059433763,-0.7838697902807743);
\draw [line width=.8pt,color=red] (5.403456059433763,-0.7838697902807743)-- (5.9196990645862995,0.04202586792561247);
\draw [line width=0.4pt,color=red] (7.337468452009789,1.5172699485617462)-- (6.015090593598238,2.346267707808551);
\draw [line width=0.4pt,color=red] (3.614600171992503,0.5864350017498108)-- (4.2870642955230185,1.7015897898486292);
\draw [line width=0.4pt,color=red] (6.77798328814177,-0.44949022277130357)-- (7.58130547427894,-0.9095310456872623);
\draw [line width=0.4pt,color=red] (6.77798328814177,-0.44949022277130357)-- (7.451218600423436,1.9171362467786506);
\draw [line width=0.8pt,color=red] (4.080470903496167,0.2295672963476746)-- (3.614600171992503,0.5864350017498108);
\draw [line width=.8pt,color=blue] (6.700820665130815,-2.561449698410666)-- (6.424109578695687,0.729969614812759);
\draw [line width=.8pt,color=blue] (6.424109578695687,0.729969614812759)-- (7.003983369841884,2.6287745626405727);
\draw [line width=.8pt,color=blue] (7.003983369841884,2.6287745626405727)-- (5.564724139554996,1.2173629393201004);
\draw [line width=.8pt,color=blue] (5.564724139554996,1.2173629393201004)-- (3.5197824983105828,0.9024619769552411);
\draw [line width=.8pt,color=blue] (3.5197824983105828,0.9024619769552411)-- (5.191040470379476,-0.07588937388178962);
\draw [line width=.8pt,color=blue] (5.191040470379476,-0.07588937388178962)-- (6.700820665130815,-2.561449698410666);
\draw [line width=0.8pt,color=blue] (6.424109578695687,0.729969614812759)-- (6.552856528789453,1.1515533276626977);
\draw [line width=0.8pt,color=blue] (5.564724139554996,1.2173629393201004)-- (4.992838951823879,1.1292982276693815);
\draw [line width=0.8pt,color=blue] (5.191040470379476,-0.07588937388178962)-- (5.324316060520116,-0.2953017891680517);
\draw [line width=0.4pt,color=red] (5.403456059433763,-0.7838697902807743)-- (4.936094580643864,-1.5315638264090725);
\begin{small}
\draw[color=black] (8.92,1.73) node {$C_5$};
\draw [fill=white] (7.451218600423436,1.9171362467786506) circle (1.5pt) node[right,yshift=-4.5pt,xshift=-2.5pt] {$A_5$};
\draw[color=black] (7.43,2.45) node {$C_6$};
\draw [fill=white] (6.015090593598238,2.346267707808551) circle (1.5pt) node[left] {$A_6$};
\draw[color=black] (4.4,2.819160978832141) node {$C_1$};
\draw [fill=white] (4.2870642955230185,1.7015897898486292) circle (1.5pt) node[left,xshift=1pt,yshift=5pt] {$A_1$};
\draw[color=black] (3.55,1.5) node {$C_2$};
\draw [fill=white] (4.080470903496167,0.2295672963476746) circle (1.5pt)  node[below] {$A_2$};
\draw[color=black] (3.3,-.6) node {$C_3$};
\draw [fill=white] (6.479425715338928,0.071995975859382) circle (1.5pt) node[right,xshift=-2pt,yshift=4pt] {$B_4$};
\draw [fill=white] (6.552856528789453,1.1515533276626977) circle (1.5pt) node[right,xshift=-3pt,yshift=-5.5pt] {$B_5$};
\draw [fill=white] (5.958681666123824,1.6036979513825151) circle (1.5pt) node[right] {$B_6$};
\draw [fill=white] (4.586061391366323,0.27826426910604585) circle (1.5pt) node[above] {$B_2$};
\draw [fill=white] (5.324316060520116,-0.2953017891680517) circle (1.5pt) node[left] {$B_{3}$};
\draw [fill=white] (4.992838951823879,1.1292982276693815) circle (1.5pt) node[left,xshift=4pt,yshift=-6pt] {$B_1$};
\draw[color=black] (5.3,-2.4) node {$C_4$};
\draw [fill=white] (7.58130547427894,-0.9095310456872623) circle (1.5pt) node[right,xshift=-2pt,yshift=5pt] {$A_4$};
\draw [fill=white] (4.936094580643864,-1.5315638264090725) circle (1.5pt)  node[left,yshift=5pt,xshift=2pt] {$A_3$};
\draw [fill=red] (5.9196990645862995,0.04202586792561247) circle (1.5pt) node[above,red] {$X_4$};
\draw [fill=red] (6.77798328814177,-0.44949022277130357) circle (1.5pt) node[below,red] {$X_5$};
\draw [fill=red] (7.337468452009789,1.5172699485617462) circle (1.5pt) node[right,red,xshift=-2pt,yshift=-3pt] {$X_6$};
\draw [fill=red] (5.0432206312834635,2.9555323108971203) circle (1.5pt) node[above,red] {$X_1$};
\draw [fill=red] (3.614600171992503,0.5864350017498108) circle (1.5pt) node[left,red] {$X_2$};
\draw [fill=red] (5.403456059433763,-0.7838697902807743) circle (1.5pt) node[left,red] {$X_3$};
\draw [fill=blue] (6.700820665130815,-2.561449698410666) circle (1.5pt) node[blue,below,xshift=1pt] {$X'_4$};
\draw [fill=blue] (6.424109578695687,0.729969614812759) circle (1.5pt) node[left,blue,xshift=1pt] {$X'_5$};
\draw [fill=blue] (7.003983369841884,2.6287745626405727) circle (1.5pt) node[above,blue] {$X'_6$};
\draw [fill=blue] (5.564724139554996,1.2173629393201004) circle (1.5pt) node[below,blue] {$X'_1$};
\draw [fill=blue] (3.5197824983105828,0.9024619769552411) circle (1.5pt) node[left,blue] {$X'_2$};
\draw [fill=blue] (5.191040470379476,-0.07588937388178962) circle (1.5pt) node[above,blue,xshift=6pt,yshift=-2pt] {$X'_3$};
\end{small}
\end{tikzpicture}
\caption{Theorem~\ref{thm-2} for six circles. If the red polygon through the pivots $A_i$ 
with a starting point $X_1$ on $C_1$ closes, then the blue polygon through the pivots $B_i$  closes for every starting point $X_1'$ on $C_1$.}\label{fig-thm2}
\end{center}
\end{figure}
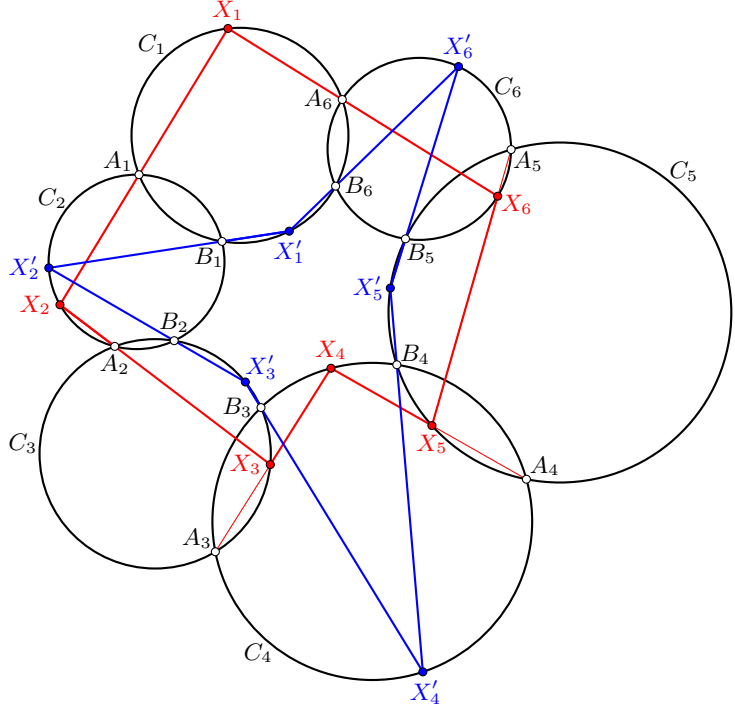

Before we get to the proof, let us introduce a useful concept
which will allow us to formulate an explicit condition on the circle chain for the polygons constructed in this way to close.
\begin{definition}\label{def-mu}
Let $C_1\mathbin{\intersect} C_2$ be two circles with centers $M_1$ and $M_2$, and $A$ a common point of $C_1$ and $C_2$.
Let $X$ be a point on $C_1$ and $\varphi_A(X)$ on $C_2$. Let $X'$ be the point on $C_2$ such that $M_1X$ and $M_2X'$ are
parallel in the sense that the vector $\overrightarrow{M_2X'}$ equals $\lambda \overrightarrow{M_1X}$ for a $\lambda>0$ (see Figure~\ref{fig-transfer}).
Then $\mu_A:=\sphericalangle X'M_2\varphi_A(X)$ is called the {\em transfer angle\/} of the map $\varphi_A$.
\end{definition}
Notice that the transfer angle is well defined as it does not depend on the choice of the point $X$. It turns out that
the transfer angle can be computed in the following way.
\begin{lemma}\label{lem-transfer}
Let $C_1\mathbin{\intersect} C_2$ be two circles with centers $M_1$ and $M_2$, and $A,B$ the intersection points of $C_1$ and $C_2$.
Let $\delta_A:=\sphericalangle AM_1B, \gamma_A:=\sphericalangle BM_2A$. Then the transfer angle $\mu_A$ of the map $\varphi_A$
is given by
$$
\mu_A=\pi-\frac12(\delta_A+\gamma_A).
$$
\end{lemma}
\begin{proof}
We choose the point $X$ on $C_1$ in such a way that the point $X'$ in Definition~\ref{def-mu} agrees with $A$ (see Figure~\ref{fig-transfer}). 
Then we have the following for the green angles in Figure~\ref{fig-transferformula}:
\begin{align*}
\sphericalangle \varphi_A(X) A M_2 &= \frac12(\pi-\mu_A)\\
\sphericalangle M_2AB &= \frac12(\pi-\gamma_A)\\
\sphericalangle BAM_1 &= \frac12(\pi-\delta_A)\\
\sphericalangle M_1AX &= \frac12(\pi-\mu_A)
\end{align*}
In the last line we have used that $M_1X$ and $M_2A$ are parallel (see Figure~\ref{fig-transferformula}).
Adding these four angles yields
$$
\frac12(\pi-\mu_A)+  \frac12(\pi-\gamma_A)+\frac12(\pi-\delta_A)+\frac12(\pi-\mu_A)=\pi,
$$
from which the formula for $\mu_A$ follows immediately.
\end{proof}
The reader is invited to also consider the situation when $M_1$ and $M_2$ lie on the same side of the line $AB$, as well as the case of touching circles with $A=B$.
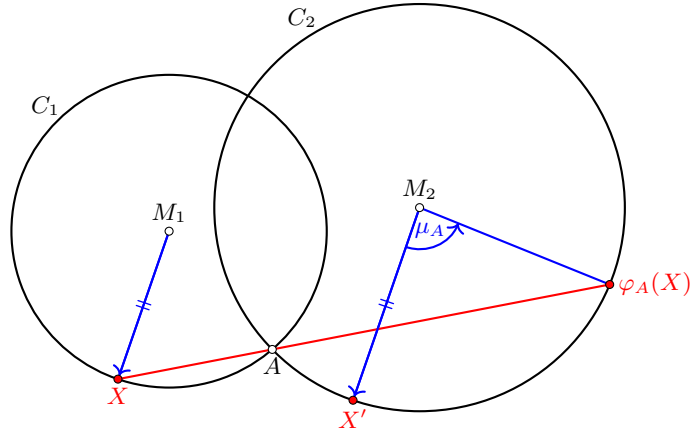
\begin{figure}[h!]
\begin{center}
\begin{tikzpicture}[line cap=round,line join=round,x=40,y=40]
\draw [shift={(7.599648342850498,1.8913361881843633)},line width=0.8pt,color=blue,->] (-108.91526549794216:0.39200351709025616) arc (-108.91526549794216:-22.182185369702875:0.39200351709025616);
\draw [line width=0.8pt] (5.256314055505859,1.6690068835868357) circle (1.4749238211549087);
\draw [line width=0.8pt] (7.599648342850498,1.8913361881843633) circle (1.9204375473422888);
\draw [line width=0.8pt,color=blue] (5.256314055505859,1.6690068835868357)-- (4.778188774370005,0.27373038993084453);
\draw [line width=0.8pt,color=blue] (7.599648342850498,1.8913361881843633)-- (9.377950504474585,1.1662694359434558);
\draw [line width=0.8pt,color=blue] (7.599648342850498,1.8913361881843633)-- (6.977101110880309,0.0746041448691363);
\draw [line width=0.8pt,color=red] (4.778188774370005,0.27373038993084453)-- (9.377950504474585,1.1662694359434558);
\draw [->,line width=0.8pt,color=blue,shorten >=1.5pt,] (5.256314055505859,1.6690068835868357) -- (4.778188774370005,0.27373038993084453);
\draw [->,line width=0.8pt,color=blue,shorten >=1.5pt,] (7.599648342850498,1.8913361881843633) -- (6.977101110880309,0.07460414486913636);

\draw [line width=0.5pt,color=blue] (5.084639401404729,0.9741167330410532) -- (4.965711960403113,1.0148701008661436);
\draw [line width=0.5pt,color=blue] (5.06879086947275,0.9278671726515352) -- (4.949863428471134,0.9686205404766257);
\draw [line width=0.5pt,color=blue] (7.355762713332201,0.9857182628089628) -- (7.2368352723305875,1.0264716306340549);
\draw [line width=0.5pt,color=blue] (7.3399141814002204,0.9394687024194449) -- (7.220986740398606,0.980222070244537);

\begin{small}
\draw [fill=white] (5.256314055505859,1.6690068835868357) circle (1.5pt) node[above] {$M_1$};
\draw[color=black] (4.1,2.849773096093096) node {$C_1$};
\draw [fill=white] (7.599648342850498,1.8913361881843633) circle (1.5pt) node[above] {$M_2$};
\draw [fill=white] (6.221573237623741,0.5538051128084196) circle (1.5pt) node[below] {$A$};
\draw[color=black] (6.5,3.68) node {$C_2$};
\draw [fill=red] (4.778188774370005,0.27373038993084453) circle (1.5pt) node[below,red] {$X$};
\draw [fill=red] (9.377950504474585,1.1662694359434558) circle (1.5pt)  node[right,red] {$\varphi_A(X)$};
\draw[color=blue] (7.7,1.7129628965313413) node[xshift=0pt,yshift=-1.5pt] {$\mu_A$};
\draw [fill=red] (6.977101110880309,0.0746041448691363) circle (1.5pt) node[below,red] {$X'$};
\end{small}
\end{tikzpicture}
\caption{The transfer angle $\mu_A$ of the map $\varphi_A$.}\label{fig-transfer}
\end{center}
\end{figure}

\begin{figure}[h!]
\begin{center}
\begin{tikzpicture}[line cap=round,line join=round,x=33,y=33]
\draw [shift={(6.070437661933166,0.38856168303488303)},line width=0.8pt,color=darkgreen,fill=darkgreen!12,->](0,0) -- (-7.602420116216897:0.7822783777527215) arc (-7.602420116216897:33.91744270025134:0.7822783777527215) ;
\draw [shift={(6.070437661933166,0.38856168303488303)},line width=0.8pt,color=darkgreen,fill=darkgreen!12,->] (0,0) -- (33.91744270025134:0.6705243237880469) arc (33.91744270025134:92.76303494516048:0.6705243237880469);
\draw [shift={(6.070437661933166,0.38856168303488303)},line width=0.8pt,color=darkgreen,fill=darkgreen!12,->] (0,0) --(92.76303494516048:.8) arc (92.76303494516048:130.87771706731488:.8);
\draw [shift={(6.070437661933166,0.38856168303488303)},line width=0.8pt,color=darkgreen,fill=darkgreen!12,->] (0,0) --(130.87771706731485:0.7) arc (130.87771706731485:172.3975798837831:0.7); 
\draw [shift={(4.923669828248287,1.7134675140170983)},line width=0.8pt,color=blue,->]  (-49.12228293268515:0.55) arc (-49.12228293268515:54.64835282300612:0.55);
\draw [shift={(8.28183673069715,1.8755376006445112)},line width=0.8pt,color=blue,->](151.6086271900696:0.6146472968057097) arc (151.6086271900696:213.9174427:0.614647296);

\draw [shift={(8.28183673069715,1.8755376006445112)},line width=0.8pt,color=blue,->]  (-146.082557:0.5) arc (-146.08255729974866:-49.12228293268504:0.5);
\draw [line width=0.8pt] (4.923669828248287,1.7134675140170983) circle (1.7522705057567411);
\draw [line width=0.8pt] (8.28183673069715,1.8755376006445112) circle (2.6648420630276024);

\draw [line width=0.8pt,color=blue] (8.28183673069715,1.8755376006445112)-- (6.070437661933166,0.38856168303488303);
\draw [line width=0.5pt,color=blue] (7.263828087577867,1.071030105242383) -- (7.152710730163091,1.2362814840983152);
\draw [line width=0.5pt,color=blue] (7.199563662467227,1.02781779958108) -- (7.088446305052451,1.1930691784370122);

\draw [line width=0.8pt,color=blue] (4.923669828248287,1.7134675140170983)-- (3.469561377164102,0.7357044865761249);
\draw [line width=0.5pt,color=blue] (4.284306493968903,1.1635664636992982) -- (4.173189136554127,1.3288178425552304);
\draw [line width=0.5pt,color=blue] (4.220042068858263,1.120354158037995) -- (4.108924711443487,1.2856055368939274);

\draw [line width=0.8pt,color=red] (3.469561377164102,0.7357044865761249)-- (10.025834104213839,-0.13937092880493185);
\draw [line width=0.8pt,color=blue] (8.28183673069715,1.8755376006445112)-- (10.025834104213839,-0.13937092880493185);
\draw [line width=0.8pt,color=blue] (4.923669828248287,1.7134675140170983)-- (6.070437661933166,0.38856168303488303);
\draw [line width=0.8pt,color=blue] (6.070437661933166,0.38856168303488303)-- (5.937521394009971,3.1426480232417267);
\draw [line width=0.8pt,color=blue] (5.937521394009971,3.1426480232417267)-- (4.923669828248287,1.7134675140170983);
\draw [line width=0.8pt,color=blue] (5.937521394009971,3.1426480232417267)-- (8.28183673069715,1.8755376006445112);
\begin{small}
\draw [fill=white] (4.923669828248287,1.7134675140170983) circle (1.5pt)  node[above,xshift=-5pt] {$M_1$};
\draw[color=black] (3.65,3.216586248220619) node {$C_1$};
\draw [fill=white] (8.28183673069715,1.8755376006445112) circle (1.5pt) node[above,xshift=3pt]{$M_2$};
\draw [fill=white] (6.070437661933166,0.38856168303488303) circle (1.5pt)  node[below,yshift=-1pt,xshift=-1pt] {$A$};
\draw[color=black] (6.85,4.378828409453245) node {$C_2$};
\draw [fill=red] (3.469561377164102,0.7357044865761249) circle (1.5pt) node[red,below,xshift=-3pt] {$X$};
\draw [fill=red] (10.025834104213839,-0.13937092880493185) circle (1.5pt)  node[red,below,xshift=17pt,yshift=7pt] {$\varphi_A(X)$};
\draw [fill=white] (5.937521394009971,3.1426480232417267) circle (1.5pt) node[above,yshift=1.5pt,xshift=-1pt] {$B$};
\draw[color=blue] (5.28,1.75) node {$\delta_A$};
\draw[color=blue] (7.93,1.85) node {$\gamma_A$};
\draw[color=blue] (8.25,1.57) node {$\mu_A$};
\end{small}
\end{tikzpicture}
\caption{Proof of the formula for the transfer angle $\mu_A$ of the map $\varphi_A$.}\label{fig-transferformula}
\end{center}
\end{figure}
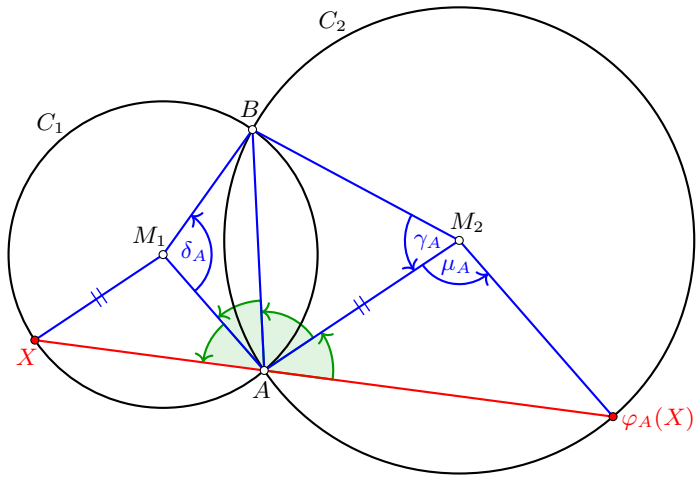
The transfer angle can also be interpreted geometrically in another way.
\begin{lemma}\label{lem-2}
Let $C_1\mathbin{\intersect} C_2$ be two circles with centers $M_1$ and $M_2$, and $A,B$ the intersection points of $C_1$ and $C_2$.
Let $t_1$ and $t_2$ be the tangents to $C_1$ and $C_2$ in $A$.
Then the transfer angle $\mu_A$ of the map $\varphi_A$
is given by $\mu_A=\sphericalangle t_2t_1$ (see Figure~\ref{fig-transferformula1}). Here, $t_1$ is oriented from $A$ towards the inside of $C_2$, 
and $t_2$ is oriented from $A$ towards the outside of $C_1$.
\end{lemma}
\begin{proof}
Recall that the central angle over a chord of a circle is twice the inscribed angle over the same chord, and the inscribed angle
over the chord equals the supplementary inscribed angle on the opposite arc, which is also the angle between the chord
and the tangent at an endpoint of the chord. In particular, we find the angle $\frac12\sphericalangle\delta_A$ between $t_1$ and the chord $AB$ (see Figure~\ref{fig-transferformula1}),
and the angle $\frac12\sphericalangle\gamma_A$ between the chord $BA$ and $t_2$. Hence, according to Lemma~\ref{lem-transfer}, the transfer angle $\mu_A$
is the angle between $t_2$ and $t_1$.
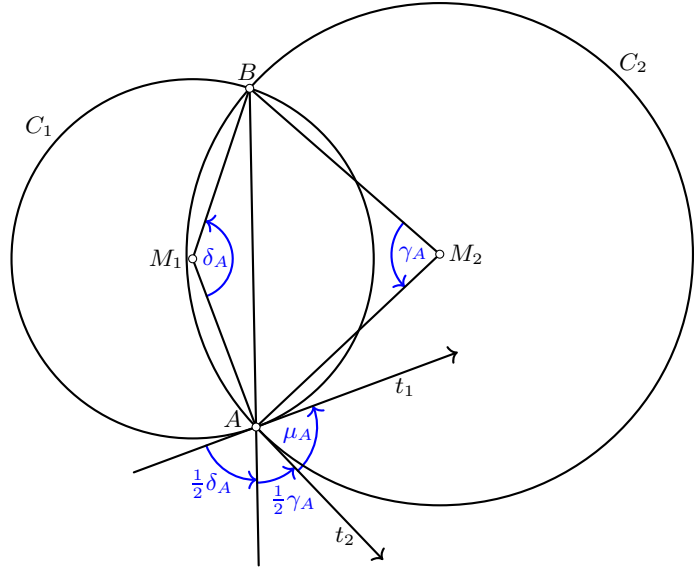
\begin{figure}[h!]
\begin{center}
\definecolor{blue}{rgb}{0.,0.,1.}
\begin{tikzpicture}[line cap=round,line join=round,x=30,y=30]
\draw [shift={(6.548564497208789,-2.5641114648411487)},line width=0.8pt,color=blue,->]  (-69.49198940446077:0.5) arc (-69.49198940446077:71.64369064604331:0.5);
\draw [shift={(9.63107451630296,-2.506223952745484)},line width=0.8pt,color=blue,->] (138.65660304470768:0.6) arc (138.65660304470768:223.49509819687484:0.6);
\draw [shift={(7.34,-4.68)},line width=0.8pt,color=blue,->] (-159.49198940446075:0.66323329) arc (-159.49198940446075:-88.92414937920873:0.66323329);
\draw [shift={(7.34,-4.68)},line width=0.8pt,color=blue,->] (-88.92414937920873:0.7) arc (-88.92414937920873:-46.50490180312506:0.7);
\draw [shift={(7.34,-4.68)},line width=0.8pt,color=blue,->](-46.50490180312504:0.7579809084438648) arc (-46.50490180312504:20.5080105955392:0.757980908);
\draw [line width=0.8pt] (6.548564497208789,-2.5641114648411487) circle (2.2590605233802497);

\draw [line width=0.8pt] (9.63107451630296,-2.506223952745484) circle (3.158215436424549);

\draw [line width=0.8pt] (7.26,-0.42)-- (7.37268568910433,-6.4205129448055525);
\draw [line width=0.8pt] (6.548564497208789,-2.5641114648411487)-- (7.26,-0.42);
\draw [line width=0.8pt] (6.548564497208789,-2.5641114648411487)-- (7.34,-4.68);
\draw [line width=0.8pt] (7.34,-4.68)-- (9.63107451630296,-2.506223952745484);
\draw [line width=0.8pt] (9.63107451630296,-2.506223952745484)-- (7.26,-0.42);
\draw [line width=0.8pt] (7.34,-4.68)-- (5.815136605408486,-5.250366069541487);
\draw [line width=0.8pt,->] (7.34,-4.68)-- (8.925513446102288,-6.3510688556019925);
\draw [line width=.8pt,->,shorten >=15mm] (7.34,-4.68)-- (11.189421195228324,-3.2401469319025615);
\begin{small}
\draw [fill=white] (7.34,-4.68) circle (1.5pt)  node[left,xshift=-2pt,yshift=3pt] {$A$};
\draw [fill=white] (7.26,-0.42) circle (1.5pt) node[above,xshift=-1pt] {$B$};
\draw[color=black] (4.64,-0.9) node {$C_1$};
\draw[color=black] (12.05,-.1) node {$C_2$};
\draw [fill=white] (6.548564497208789,-2.5641114648411487) circle (1.5pt)  node[left] {$M_1$};
\draw [fill=white] (9.63107451630296,-2.506223952745484) circle (1.5pt) node[right] {$M_2$};
\draw[color=blue] (6.84,-2.55) node {$\delta_A$};
\draw[color=blue] (9.3,-2.5) node {$\gamma_A$};
\draw[color=blue] (6.78,-5.4) node {$\frac12\delta_A$};
\draw[color=blue] (7.8,-5.6) node {$\frac12\gamma_A$};
\draw[color=blue] (7.8672919437560775,-4.78) node {$\mu_A$};

\draw[color=black] (9.2,-4.18) node {$t_1$};

\draw[color=black] (8.45,-6.05) node {$t_2$};

\end{small}
\end{tikzpicture}
\caption{Lemma~\ref{lem-2}: The transfer angle $\mu_A$ of the map $\varphi_A$.}\label{fig-transferformula1}
\end{center}
\end{figure}
\end{proof}

Using the notion of the transfer angle we can now formulate the following closing condition.
\begin{theorem}\label{thm-closingcondition}
Let $C_1\mathbin{\intersect} C_2 \mathbin{\intersect} C_3\mathbin{\intersect} \cdots \mathbin{\intersect}C_n\mathbin{\intersect}C_1$ be a closed
chain of circles with centers $M_1,\ldots,M_n$. Let $A_i,B_i$ be the intersection points of $C_i$ and $C_{i+1}$, where $A_i=B_i$ if $C_i$ touches $C_{i+1}$. 
Let $\delta_{A_i}:=\sphericalangle A_iM_iB_i, \gamma_{A_i}:=\sphericalangle B_iM_{i+1}A_i$. Then the map
$\varphi_A:=\varphi_{A_n}\circ\cdots\circ\varphi_{A_2}\circ\varphi_{A_1}$ is the identity map on $C_1$ if and only if
the sum of all transfer angles is a multiple of $2\pi$, i.e.,
$$
n\pi -\frac12\sum_{i=1}^n(\delta_{A_i}+\gamma_{A_i}) = 2k\pi
$$
for an integer $k$.
\end{theorem}
\begin{proof}
Clearly, the map $\varphi_A$ is the identity on $C_1$ if and only if the sum of 
all transfer angles is a multiple of $2\pi$. Using the formula from Lemma~\ref{lem-transfer} for the transfer angles $\mu_{A_i}$ we
get the closing condition stated in the theorem.
\end{proof}
Observe that the situation of two circles with centers $M_1$ and $M_2$, intersecting in the points $A$ and $B$ like in Figure~\ref{fig-transferformula} 
is mirror symmetric with respect to the line $M_1M_2$. Hence for the transfer angles of $\varphi_A$ and $\varphi_B$ we have
$$
\mu_B=-\mu_A.
$$
Now, if the sum of the transfer angles $\mu_{A_i}$ is a multiple of $2\pi$, then the same is true for the sum
of the transfer angles $\mu_{B_i}=-\mu_{A_i}$. This proves Theorem~\ref{thm-2}. Another application of the criterion in Theorem~\ref{thm-closingcondition}
is the following.
\begin{corollary}\label{cor-fivecircles}
Let $l_1, l_2, \ldots, l_n,l_{n+1}=l_1, l_{n+2}=l_2$ be a set of lines such that $l_{i-1},l_{i}, l_{i+1}$ form a triangle with circumcircle $C_i$
(see Figure~\ref{fig-fivecircles}). Let $A_i$ denote the intersection of $l_{i}$ and $l_{i+1}$. Then, $\varphi:=\varphi_{A_n}\circ\varphi_{A_{n-1}}\circ\ldots
\circ\varphi_{A_1}$ is the identity map on $C_1$, i.e., the resulting $n$-gon $X_1$, $X_{i+1}=\varphi_{A_i}(X_i)$ closes for any
initial point $X_1$ on $C_1$.
\end{corollary}
Note that the situation of Corollary~\ref{cor-fivecircles}  corresponds exactly to that in Morley's Five Circles theorem~\cite{morley,morley2,hhl} and of Miquel's Pentagon theorem~\cite{miquel, nhpentagon}.
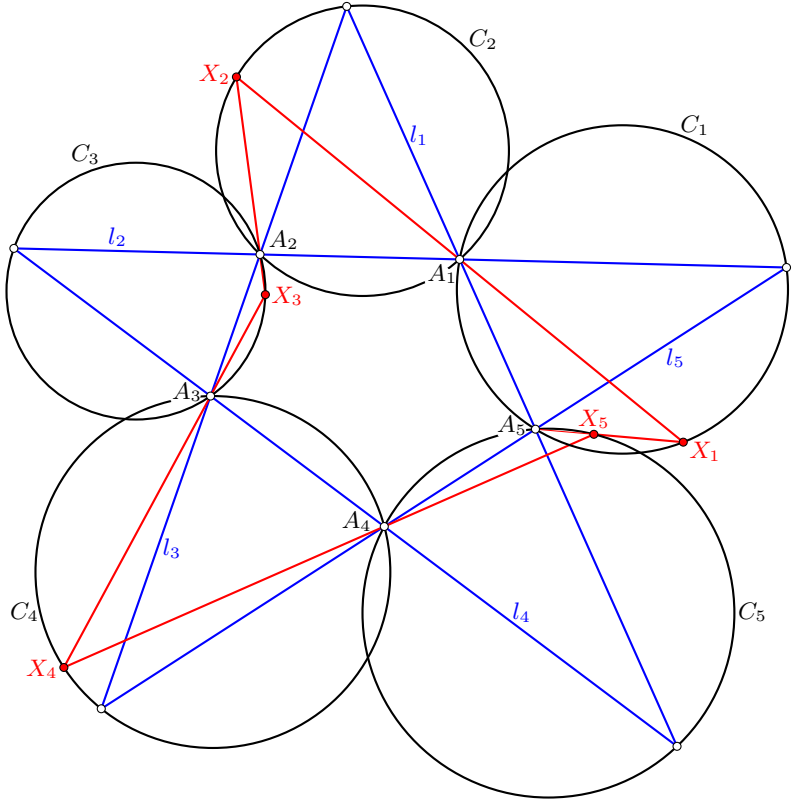
\begin{figure}
\begin{center}
\begin{tikzpicture}[line cap=round,line join=round,x=12,y=12]
\draw[line width=.8pt,color=blue] (-3.5076794741867094,-13.039190844372085) -- (17.85982895558124,0.8358146295030996) -- (-6.23276325001621,1.4359859655855958) -- (14.446504433121106,-14.219504822778118) -- (4.146182542718997,9.044535168162536) -- cycle;
\draw[line width=.8pt,color=red] (14.640520013605277,-4.656391920532995) -- (0.7070868254103599,6.826624237115578) -- (1.6093595709876263,-0.017861169381211097) -- (-4.674105517688726,-11.740757360603393) -- (11.853425465831393,-4.409200242649938) -- cycle;
\draw [line width=0.8pt] (4.638737813300063,4.505489570747567) circle (4.565692239955449);
\draw [line width=0.8pt] (-2.4260121738788225,0.09246684260238183) circle (4.03687966002143);
\draw [line width=0.8pt] (-0.026702204865289382,-8.73897660395168) circle (5.532544194405212);
\draw [line width=0.8pt] (10.436576294275646,-10.031560749280048) circle (5.798137566620178);
\draw [line width=0.8pt] (12.7436676228852,0.14508359734482768) circle (5.162578439206573);
\draw [line width=0.4pt,color=red] (11.853425465831393,-4.409200242649938)-- (10.03137124639526,-4.247599467370873);
\begin{small}
\draw [fill=white] (5.317694604689856,-7.308428455491241) circle (1.5pt) node[left,xshift=-1.5pt,yshift=2pt] {$A_4$};

\draw [fill=white,white] (9.2,-4.247599467370873) circle (4.5pt) ;
\draw [fill=white,white] (9.5,-4.247599467370873) circle (2.7pt) ;
\draw [fill=white] (10.03137124639526,-4.247599467370873) circle (1.5pt)  node[left,yshift=1.5pt] {$A_5$};
\draw[color=blue] (14.4,-2) node {$l_5$};

\draw [fill=white,white] (7.,.5) circle (5pt);
\draw [fill=white] (7.668244420693079,1.0896975182725543) circle (1.5pt) node[below,xshift=-6.5pt,yshift=1.4pt] {$A_1$};
\draw[color=blue] (6.4,5) node {$l_1$};
\draw [fill=white] (1.4429125843792387,1.24477695903482) circle (1.5pt) node[anchor=south west,yshift=-1.4pt] {$A_2$};
\draw[color=blue] (-3,1.8) node {$l_2$};
\draw[color=blue] (-1.3,-8) node {$l_3$};

\draw [fill=white,white] (-0.99,-3.18) circle (4.5pt) ;
\draw [fill=white,white] (-0.7,-3.17) circle (3.3pt) ;
\draw [fill=white] (-0.09997270428324008,-3.206917611409948) circle (1.5pt) node[left,yshift=1.5pt] {$A_3$};
\draw[color=blue] (9.6,-10) node {$l_4$};
\draw [fill=white] (4.146182542718997,9.044535168162536) circle (1.5pt);
\draw[color=black] (8.4,8) node {$C_2$};
\draw [fill=white] (-6.23276325001621,1.4359859655855958) circle (1.5pt);
\draw[color=black] (-4,4.4) node {$C_3$};
\draw [fill=white] (-3.5076794741867094,-13.039190844372085) circle (1.5pt);
\draw[color=black] (-5.9,-10) node {$C_4$};
\draw [fill=white] (14.446504433121106,-14.219504822778118) circle (1.5pt);
\draw[color=black] (16.8,-10) node {$C_5$};
\draw [fill=white] (17.85982895558124,0.8358146295030996) circle (1.5pt);
\draw[color=black] (15,5.3) node {$C_1$};
\draw [fill=red] (14.640520013605277,-4.656391920532995) circle (1.5pt) node[red,right,yshift=-3.5pt,xshift=-1] {$X_1$};
\draw [fill=red] (0.7070868254103599,6.826624237115578) circle (1.5pt)  node[red,left,xshift=1pt,yshift=1pt]  {$X_2$};
\draw [fill=red] (1.6093595709876263,-0.017861169381211097) circle (1.5pt) node[red,right,xshift=-1pt,yshift=-.5pt]  {$X_3$};
\draw [fill=red] (-4.674105517688726,-11.740757360603393) circle (1.5pt) node[red,left,xshift=1pt,yshift=-.5pt]  {$X_4$};
\draw [fill=red] (11.853425465831393,-4.409200242649938) circle (1.5pt) node[red,above,yshift=-1pt]  {$X_5$};
\end{small}
\end{tikzpicture}
\caption{Illustration for Corollary~\ref{cor-fivecircles} for $n=5$ lines $l_1,\ldots,l_5$.
The red polygon through the pivots $A_1,A_2,\ldots, A_n$ closes for every position of the starting point $X_1$ on $C_1$.}\label{fig-fivecircles}
\end{center}
\end{figure}
\begin{proof}
Let $A_{i,i+2}$ denote the intersection of the lines $l_i$ and $l_{i+2}$, i.e., $C_i$ is the circumcircle
of the triangle $A_{i-1}A_iA_{i-1,i+1}$. 
Consider the point $A_i$ with the transfer angle $\mu_{A_i}$ of the map $\varphi_{A_i}$. According to Lemma~\ref{lem-2} this transfer angle
is the angle between the tangents $t_{i+1}$ and $t_i$ (see Figure~\ref{fig-proof-fivecircles}). Let
$\varepsilon_i:= \sphericalangle A_{i-1}A_{i-1,i+1}A_i$. This angle equals the angle between the line $l_i$ and the tangent $t_i$ 
(see Figure~\ref{fig-proof-fivecircles}). Similarly, the angle $\varepsilon_{i+1}=\sphericalangle A_iA_{i,i+2}A_{i+1}$ appears also
as angle between the tangent $t_{i+1}$ and the line $l_{i+1}$. The angle $\omega_i$ between the lines $l_{i+1}$ and $l_i$
is the exterior angle of the polygon $A_1A_2\ldots A_n$ in the vertex $A_i$. 
We have 
$\mu_{A_i}=\omega_i+\epsilon_i+\epsilon_{i+1}$. Observe that 
$\varepsilon_{i}=\pi-(\omega_{i-1}+\omega_{i})$, and
$\varepsilon_{i+1}=\pi-(\omega_{i}+\omega_{i+1})$. Hence, we get 
$\mu_{A_i}=2\pi-(\omega_{i-1}+\omega_{i}+\omega_{i+1})$. Taking the sum over all transfer angles results in
$$
\sum_{i=1}^n\mu_{A_i}=2n\pi -3\sum_{i=1}^n \omega_i.
$$
Since the sum of the exterior angles $\sum_{i=1}^n \omega_i$ of a polygon is a multiple of $2\pi$, 
 the closing condition of Theorem~\ref{thm-closingcondition} is satisfied.
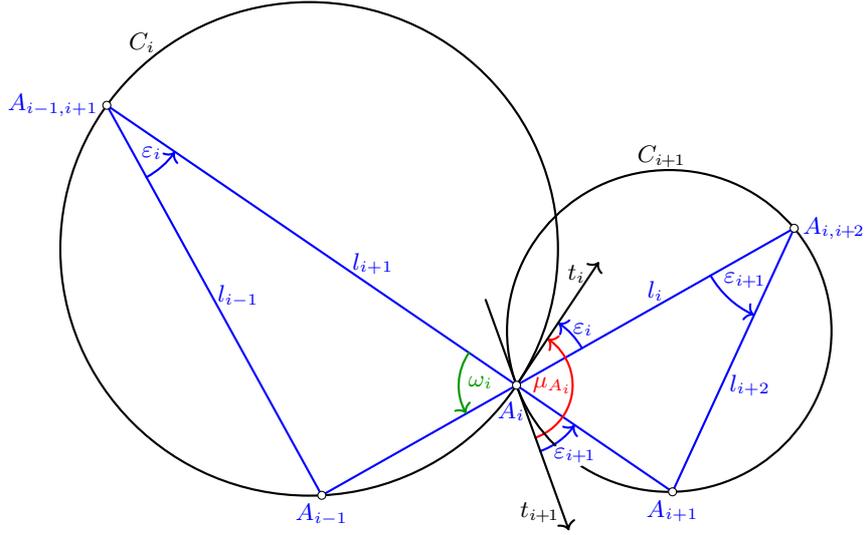
\begin{figure}[h!]
\begin{center}
\begin{tikzpicture}[line cap=round,line join=round,x=12,y=12]
\draw[line width=0.8pt,color=blue] (8.40381740677802,1.53967312) -- (23.13540,9.936679) -- (19.342023,1.65016) -- (1.70094367,13.8037223) -- cycle;
\draw [shift={(1.7009436706780645,13.803722328087488)},line width=0.8pt,->,blue] (-61.341342:2.5780) arc (-61.34134:-34.5643:2.57802) ;
\draw [shift={(23.13540803556913,9.936679788029819)},line width=0.8pt,->,blue] (-150.31685:3) arc (-150.31685:-114.59725:3) ;
\draw [shift={(14.477364714794994,5.001595095188559)},line width=0.8pt,->,red] (-70.28393291499289:1.75) arc (-70.28393:56.460154928:1.75);
\draw [shift={(14.477364714794994,5.001595095188559)},line width=0.8pt,->,blue] (-70.2839329:2.2) arc (-70.28393:-34.5643272:2.2) ;
\draw [shift={(14.477364714794994,5.001595095188562)},line width=0.8pt,->,blue] (29.68314017:2.35) arc (29.68314017:56.4601549:2.35) ;

\draw [shift={(14.477364714794994,5.001595095188559)},line width=0.8pt,->,darkgreen](145.4356727:1.8) arc (145.435672:209.683148:1.8);
\draw [line width=0.8pt] (8.010452203350463,9.288421581589764) circle (7.758726619438833);
\draw [line width=0.8pt] (19.240018671421453,6.708379027154118) circle (5.059247384837539);

\draw [line width=0.8pt,->] (14.477364714794994,5.001595095188559)-- (17.04603217931246,8.876570631689395);
\draw [line width=0.8pt,->] (13.514928021629942,7.687203296003002)-- (16.11249131385353,0.43889561065390126);
\begin{small}
\draw [fill=white] (8.40381740677802,1.5396731296188835) circle (1.5pt) node[blue,below] {$A_{i-1}$};
\draw [fill=white] (23.13540803556913,9.936679788029819) circle (1.5pt) node[blue,right] {$A_{i,i\hspace{-.5pt}+\hspace{-.7pt}2}$};
\draw [fill=white] (19.34202344865542,1.6501600593348167) circle (1.5pt) node[blue,below] {$A_{i+1}$};
\draw [fill=white] (1.7009436706780645,13.803722328087488) circle (1.5pt) node[blue,left] {$A_{i-1,i+1}$};
\draw[color=blue] (18.808003288361743,7.96) node {$l_{i}$};
\draw[color=blue] (21.754321414119964,5.) node {$l_{i+2}$};
\draw[color=blue] (10,8.831810490870485) node {$l_{i+1}$};
\draw[color=blue] (5.8,7.800599146855108) node {$l_{i-1}$};
\draw [fill=white] (14.477364714794994,5.001595095188559) circle (1.5pt) node[blue,below,xshift=-2pt,yshift=-3pt] {$A_{i}$};
\draw[color=black] (2.8,15.77) node {$C_{i}$};
\draw[color=black] (19,12.2) node {$C_{i\hspace{-.5pt}+\hspace{-.5pt}1}$};
\draw[color=black] (21.6,8.3) node[blue] {$\epsilon_{i\hspace{-.4pt}+\hspace{-.7pt}1}$};
\draw[color=black] (16.35,8.5) node {$t_{i}$};
\draw[color=black] (15.2,1.) node {$t_{i\hspace{-.4pt}+\hspace{-.7pt}1}$};
\draw[color=black] (15.6,4.99) node[red] {$\mu_{A_i}$};

\draw [fill=white,white] (16.3,2.8) circle (4.5pt);
\draw [fill=white,white] (15.9,2.8) circle (4.pt);

\draw[color=black] (16.3,2.8) node[blue] {$\epsilon_{i\hspace{-.4pt}+\hspace{-.7pt}1}$};

\draw[color=blue] (16.55,6.75) node {$\epsilon_{i}$};
\draw[color=blue] (3.1,12.3) node {$\epsilon_{i}$};

\draw[color=black] (13.35731475570903,5.1) node[darkgreen] {$\omega_{i}$};
\end{small}
\end{tikzpicture}
\caption{Proof of Corollary~\ref{cor-fivecircles}.}\label{fig-proof-fivecircles}
\end{center}
\end{figure}
\end{proof}

We will see in Section~\ref{sec-4} that for $n=4$ Corollary~\ref{cor-fivecircles} together 
with Lemma~\ref{lem-lighthouse} implies Steiner's quadrilateral theorem~\cite{steiner2},
even in an extended version (see Corollary~\ref{cor-steiner}).

\section{Neat special cases}
\begin{theorem}\label{thm-ab}
Let $C_1\mathbin{\intersect} C_2 \mathbin{\intersect} C_3\mathbin{\intersect} \cdots \mathbin{\intersect}C_n\mathbin{\intersect}C_1$ be a closed
chain of circles. Let $A_i,B_i$ be the intersection points of $C_i$ and $C_{i+1}$, where $A_i=B_i$ if $C_i$ touches $C_{i+1}$. 
Let $\varphi_A:=\varphi_{A_n}\circ\cdots\circ\varphi_{A_2}\circ\varphi_{A_1}$ and
$\varphi_B:=\varphi_{B_n}\circ\cdots\circ\varphi_{B_2}\circ\varphi_{B_1}$. Then 
$\varphi_B\circ\varphi_A(X)=X$ holds for all points $X\in C_1$ (see Figure~\ref{fig-ab}).
\end{theorem}
\begin{proof}
The claim follows immediately from Theorem~\ref{thm-closingcondition} if we use the fact that the sum of the transfer angles
$\mu_{A_i}+\mu_{B_i}=0$ for all $i$. Therefore the total sum of the transfer angles along the chain vanishes.
\end{proof}
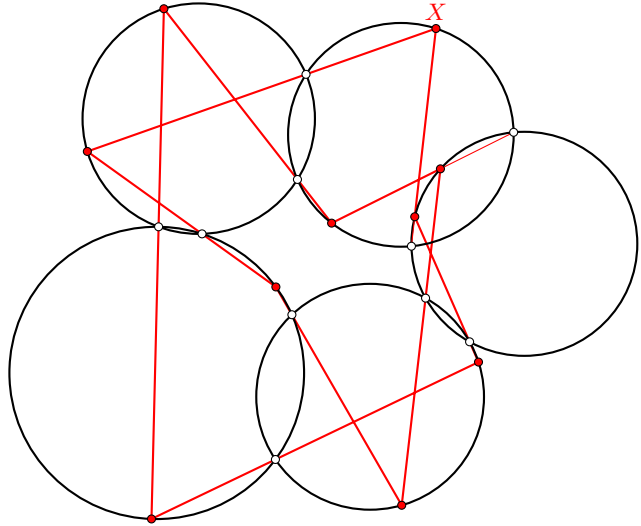
\begin{figure}[h!]
\begin{center}
\begin{tikzpicture}[line cap=round,line join=round,x=55,y=55]
\draw[line width=.8pt,color=red] (7.019667591018827,0.16582876428965498) -- (9.24439279619086,1.242011206873463) -- (8.810708211574527,2.238892943280372) -- (8.955056070761842,3.5301236770449336) -- (6.58326651841914,2.686815983771707) -- (7.865949362448734,1.7573368279404074) -- (8.722019233142806,0.25964130845981814) -- (8.98549890092734,2.5676023648316653) -- (8.24516033402484,2.195101669588272) -- (7.10385586447076,3.665830503595402) -- cycle;
\draw [line width=0.8pt] (8.507181085786682,1.003951723542948) circle (0.7746956973939061);
\draw [line width=0.8pt] (9.556418916231944,2.0531895539882123) circle (0.7684856562326098);
\draw [line width=0.8pt] (8.716387887906961,2.800781143533427) circle (0.7674001776754021);
\draw [line width=0.8pt] (7.341115179678269,2.9117497362279074) circle (0.7905250067937619);
\draw [line width=0.8pt] (7.055356148792899,1.1676099284452897) circle (1.0024166668671326);

\draw [line width=0.4pt,color=red] (8.810708211574527,2.238892943280372)-- (8.788109335558955,2.036739870667267);
\draw [line width=0.4pt,color=red] (8.98549890092734,2.5676023648316653)-- (9.483589975260754,2.8182164484123744);

\begin{small}
\draw [fill=white] (8.88424784483196,1.680689547802625) circle (1.5pt);
\draw [fill=white] (9.183918910046357,1.3810184825882286) circle (1.5pt);
\draw [fill=white] (7.8629199286930955,0.5737413273167934) circle (1.5pt);
\draw [fill=white] (9.483589975260754,2.8182164484123744) circle (1.5pt);
\draw [fill=white] (8.788109335558955,2.036739870667267) circle (1.5pt);
\draw [fill=white] (8.070854953535738,3.2157392900233086) circle (1.5pt);
\draw [fill=white] (8.012674744330221,2.494691541858395) circle (1.5pt);
\draw [fill=white] (7.067874245471225,2.169948429785313) circle (1.5pt);
\draw [fill=white] (7.974862533948821,1.566793071464255) circle (1.5pt);
\draw [fill=red] (7.019667591018827,0.16582876428965498) circle (1.5pt);
\draw [fill=red] (9.24439279619086,1.242011206873463) circle (1.5pt);
\draw [fill=red] (8.810708211574527,2.238892943280372) circle (1.5pt);
\draw [fill=white] (7.363351754622448,2.1215375361822244) circle (1.5pt);
\draw [fill=red] (8.955056070761842,3.5301236770449336) circle (1.5pt) node[red,above] {$X$};
\draw [fill=red] (6.58326651841914,2.686815983771707) circle (1.5pt);
\draw [fill=red] (7.865949362448734,1.7573368279404074) circle (1.5pt);
\draw [fill=red] (8.722019233142806,0.25964130845981814) circle (1.5pt);
\draw [fill=red] (8.98549890092734,2.5676023648316653) circle (1.5pt);
\draw [fill=red] (8.24516033402484,2.195101669588272) circle (1.5pt);
\draw [fill=red] (7.10385586447076,3.665830503595402) circle (1.5pt);
\end{small}
\end{tikzpicture}
\caption{Illustration for Theorem~\ref{thm-ab}. The red polygon closes for every position of the starting point $X$ on the circle.}\label{fig-ab}
\end{center}
\end{figure}

Let us now consider chains of touching circles. We will use the notation $C_1\mathbin{\touch} C_2$ for two touching circles $C_1$ and $C_2$.
Then we have the following.
\begin{theorem}\label{thm-touching}
Let $C_1\mathbin{\touch} C_2 \mathbin{\touch} C_3\mathbin{\touch} \cdots \mathbin{\touch}C_n\mathbin{\touch}C_1$ be a closed
chain of touching circles. Let $A_i$ be the common point of $C_i$ and $C_{i+1}$ and
 $\varphi=\varphi_{A_n}\circ\cdots\circ\varphi_{A_2}\circ\varphi_{A_1}$. Then the following holds: 
If $n$ is even, then $\varphi$ is the identity map on $C_1$. 
If $n$ is odd, then $\varphi\circ\varphi$ is the identity map on $C_1$ (see Figure~\ref{fig-touching}).
\end{theorem}
Note that this situation of a chain of touching circle occurs in particular in Steiner's closing theorem.

\begin{figure}[h!]
\begin{center}
\begin{tikzpicture}[line cap=round,line join=round,x=18,y=18]
\draw[line width=0.8pt,color=red] (4.743563968500389,-2.072614800665446) -- (6.952893013689016,4.3824541635110625) -- (3.490565311448319,3.375059183765999) -- (1.0073603535830142,5.623309537967394) -- (0.47694247676241214,0.32158323491683377) -- (2.440561142480428,0.7302685728108522) -- cycle;
\draw [line width=0.8pt] (2.072705033346347,-0.6792163643480142) circle (1.456696916007238);
\draw [line width=0.8pt] (5.155643216437474,-0.4936835869082452) circle (1.6318189496255953);
\draw [line width=0.8pt] (6.501119602373459,2.6514299873540264) circle (1.7890064040164408);
\draw [line width=0.8pt] (3.9266678662436973,5.04603858788007) circle (1.7269503777678348);
\draw [line width=0.8pt] (0.5283998489571511,3.78811474304257) circle (1.8966663123255578);
\draw [line width=0.8pt] (0.678964519970886,1.0956549977772765) circle (0.8);
\begin{small}
\draw [fill=white] (3.5267712382578833,-0.5917099255039995) circle (1.5pt);
\draw [fill=white] (5.797468985328758,1.006613972492653) circle (1.5pt);
\draw [fill=white] (5.19117442141191,3.869866500467226) circle (1.5pt);
\draw [fill=white] (2.307114602717334,4.446534764811715) circle (1.5pt);
\draw [fill=white] (0.6342976154078545,1.894407066717794) circle (1.5pt);
\draw [fill=white] (1.1730461145312752,0.46646235802641045) circle (1.5pt);
\draw [fill=red] (4.743563968500389,-2.072614800665446) circle (1.5pt);
\draw [fill=red] (6.952893013689016,4.3824541635110625) circle (1.5pt)  node[above,red] {$X$};
\draw [fill=red] (3.490565311448319,3.375059183765999) circle (1.5pt);
\draw [fill=red] (1.0073603535830142,5.623309537967394) circle (1.5pt);
\draw [fill=red] (0.47694247676241214,0.32158323491683377) circle (1.5pt);
\draw [fill=red] (2.440561142480428,0.7302685728108522) circle (1.5pt);
\end{small}
\end{tikzpicture}
\definecolor{red}{rgb}{1.,0.,0.}
\begin{tikzpicture}[line cap=round,line join=round,x=19,y=19]
\draw[line width=0.8pt,color=red] (2.1453339778057954,-2.1341015630945357) -- (4.722647069495119,0.7435001014018278) -- (6.0330138512271985,0.5088691328795654) -- (3.745981632863882,5.42118473481178) -- (0.32719925544748846,0.3213282827053435) -- (2.0000760888868996,0.775668834398507) -- (4.848393206238091,-1.7754156629469056) -- (5.863920718004747,3.896101297218477) -- (3.868491989057657,2.9670873326526306) -- (0.06793426556592934,5.514860959508458) -- cycle;
\draw [line width=0.8pt] (2.072705033346347,-0.6792163643480142) circle (1.456696916007238);
\draw [line width=0.8pt] (4.785520137866608,-0.5159577807725406) circle (1.2610262406855737);
\draw [line width=0.8pt] (5.948467284615967,2.2024852150490215) circle (1.695725082585253);
\draw [line width=0.8pt] (3.807236810960772,4.194136033732205) circle (1.2285767015798228);
\draw [line width=0.8pt] (0.19756676050670677,2.9180946211069005) circle (2.6);
\begin{small}
\draw [fill=white] (3.5267712382578833,-0.5917099255039995) circle (1.5pt);
\draw [fill=white] (5.2815060048684614,0.6434322310146967) circle (1.5pt);
\draw [fill=white] (4.706824496093136,3.3573907269494585) circle (1.5pt);
\draw [fill=white] (2.6489064782947525,3.784658864538873) circle (1.5pt);
\draw [fill=white] (1.3993719755193155,0.6125221921495062) circle (1.5pt);
\draw [fill=red] (2.1453339778057954,-2.1341015630945357) circle (1.5pt);
\draw [fill=red] (4.722647069495119,0.7435001014018278) circle (1.5pt);
\draw [fill=red] (6.0330138512271985,0.5088691328795654) circle (1.5pt);
\draw [fill=red] (3.745981632863882,5.42118473481178) circle (1.5pt);
\draw [fill=red] (0.32719925544748846,0.3213282827053435) circle (1.5pt);
\draw [fill=red] (2.0000760888868996,0.775668834398507) circle (1.5pt);
\draw [fill=red] (4.848393206238091,-1.7754156629469056) circle (1.5pt);
\draw [fill=red] (5.863920718004747,3.896101297218477) circle (1.5pt) node[red,above] {$X$};
\draw [fill=red] (3.868491989057657,2.9670873326526306) circle (1.5pt);
\draw [fill=red] (0.06793426556592934,5.514860959508458) circle (1.5pt);
\end{small}
\end{tikzpicture}
\caption{Illustration for Theorem~\ref{thm-touching}. On the left: For an even number of touching circles the red polygon closes for any position of the starting point $X$.
On the right: For an odd number of touching circles the red polygon closes after the second round for any position of the starting point $X$.}\label{fig-touching}
\end{center}
\end{figure}
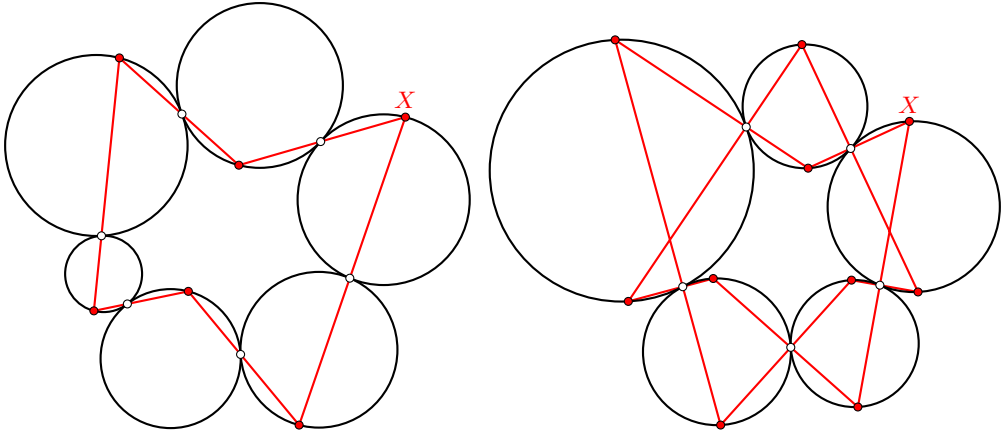

\begin{proof}
The claim follows immediately from the fact that we have $\delta_{A_i}=\gamma_{A_i}=0$ in the case of touching circles.
Then, the closing condition in Theorem~\ref{thm-closingcondition} is trivially satisfied for an even number $n$.
If $n$ is odd then apply the result to the chain  that runs through twice
$$
C_1\mathbin{\touch} C_2 \mathbin{\touch} C_3\mathbin{\touch} \cdots \mathbin{\touch}C_n\mathbin{\touch}C_1\mathbin{\touch} C_2 \mathbin{\touch} C_3\mathbin{\touch} \cdots \mathbin{\touch}C_n\mathbin{\touch} C_1
$$
and we are done.
\end{proof}

Miquel's triangle theorem~\cite[Th\'eor\`eme~I, Planche~II, Fig.~1]{miquel} turns out to be a special case of Theorem~\ref{thm-2}. To see this, consider the case of 
three circles.
\begin{corollary}\label{cor-miquel}
Let $C_1\mathbin{\intersect} C_2\mathbin{\intersect} C_3\mathbin{\intersect} C_1$ be a closed chain of three circles which all
pass through a point $B$. Let $A_i$ be the other common point of $C_i$ and $C_{i+1}$. Then, the 
map $\varphi_{A_3}\circ \varphi_{A_2}\circ \varphi_{A_1}$ is the identity map on $C_1$ (see Figure~\ref{fig-miquel} on the right). 
\end{corollary}
\begin{proof}
The points $B_i=B$ and $A_i$ are the common points of  $C_i$ and $C_{i+1}$. 
Obviously the map $\varphi_{B_3}\circ \varphi_{B_2}\circ \varphi_{B_1}$ is the identity map on $C_1$. Hence, according to Theorem~\ref{thm-2}
this is also the case for $\varphi_{A_3}\circ \varphi_{A_2}\circ \varphi_{A_1}$.
\end{proof}
Notice that the statement of Corollary~\ref{cor-miquel} is true for any number  $n\ge 3$ of
circles which pass through a common point $B$ (see Figure~\ref{fig-miquel} on the left).
\begin{figure}[h!]
\begin{center}
\begin{tikzpicture}[line cap=round,line join=round,x=26,y=26]
\draw[line width=0.8pt,color=red] (-0.147154,-1.149214285294474) -- (4.01287,1.1929519363342183) -- (-1.8670710592958077,4.428821372416705) -- cycle;
\draw[line width=0.5pt,dash pattern=on 2pt off 2pt,color=red] (0.626117942,-1.65985) -- (4.293239321051615,2.1274012) -- (-2.94831676,3.70392) -- cycle;
\draw [line width=0.8pt] (1.0990612725844464,-0.10287117046347219) circle (1.6272327499626194);
\draw [line width=0.8pt] (2.5800787401574805,2.1321188260558337) circle (1.7131670764744231);
\draw [line width=0.8pt] (-1.187452123230641,2.246286427976685) circle (2.285900409457977);
\begin{small}
\draw [fill=white] (0.98,1.52) circle (1.5pt)  node[anchor=south west,yshift=-1.4pt,xshift=-.5pt] {$B$};
\draw [fill=white] (-0.52,0.06) circle (1.5pt) node[below,xshift=-6pt,yshift=-1pt] {$A_3$};
\draw [fill=white] (2.64,0.42) circle (1.5pt) node[below,xshift=8.8pt,yshift=1.2pt]  {$A_1$};
\draw[color=black] (2.77,-1) node {$C_1$};
\draw [fill=white] (1.02,2.84) circle (1.5pt) node[above,xshift=9pt,yshift=-1.5pt]  {$A_2$};
\draw[color=black] (3.21,3.96) node {$C_2$};
\draw[color=black] (-3.7,2.6) node {$C_3$};
\draw [fill=red] (-0.14715400119318822,-1.149214285294474) circle (1.5pt);
\draw [fill=red] (4.012875650108974,1.1929519363342183) circle (1.5pt);
\draw [fill=red] (-1.8670710592958077,4.428821372416705) circle (1.5pt);
\draw [fill=red] (0.6261179429129198,-1.6598589769250764) circle (1.5pt) ;
\draw [fill=red] (4.293239321051615,2.1274012009761107) circle (1.5pt);
\draw [fill=red] (-2.9483167676432127,3.7039202592312748) circle (1.5pt);
\end{small}
\end{tikzpicture}
\definecolor{red}{rgb}{1.,0.,0.}
\begin{tikzpicture}[line cap=round,line join=round,x=30,y=30]
\draw[line width=0.8pt,color=red] (-0.027924848113163314,-1.27666361647165) -- (4.105217937929536,1.3518035953092236) -- (-0.028608244601448157,3.3458103028508397) -- (-0.7872537626249988,0.7859630556353268) -- cycle;
\draw[line width=0.5pt,dash pattern=on 2pt off 2pt,color=red] (0.5250902507448667,-1.6255148216962014) -- (4.289246768012873,2.0151312862371777) -- (-0.4685192214636319,3.215570589419363) -- (-0.7132199683322104,0.3716239197901742) -- cycle;
\draw [line width=0.8pt] (1.0990612725844464,-0.10287117046347219) circle (1.6272327499626194);
\draw [line width=0.8pt] (2.5800787401574805,2.1321188260558337) circle (1.7131670764744231);
\draw [line width=0.8pt] (0.06895297249334509,2.2082135462880803) circle (1.1417725568707313);
\draw [line width=0.8pt] (0.25989740092711205,0.7592834921981726) circle (1.0474909825071224);
\begin{small}
\draw [fill=white] (0.98,1.52) circle (1.5pt) node[below,xshift=-2pt,yshift=-2pt] {$B$};
\draw [fill=white] (-0.52,0.06) circle (1.5pt) node[below,xshift=-6.7pt,yshift=5pt] {$A_4$};
\draw [fill=white] (2.64,0.42) circle (1.5pt) node[below,xshift=9pt,yshift=1.5pt] {$A_1$};
\draw [fill=white] (1.02,2.84) circle (1.5pt) node[above,xshift=8.5pt,yshift=-2pt] {$A_2$};
\draw [fill=white] (-0.6326969457658537,1.307473870338371) circle (1.5pt)  node[left,xshift=0pt,yshift=0pt] {$A_3$};

\draw[color=black] (2.7,-1) node {$C_1$};
\draw[color=black] (4.,3.474) node {$C_2$};
\draw[color=black] (-1.13,2.883392577438442) node {$C_3$};
\draw[color=black] (.95,.3) node {$C_4$};
\draw [fill=red] (-0.027924848113163314,-1.27666361647165) circle (1.5pt);
\draw [fill=red] (4.105217937929536,1.3518035953092236) circle (1.5pt);
\draw [fill=red] (-0.028608244601448157,3.3458103028508397) circle (1.5pt);
\draw [fill=red] (-0.7872537626249988,0.7859630556353268) circle (1.5pt);
\draw [fill=red] (0.5250902507448667,-1.6255148216962014) circle (1.5pt);
\draw [fill=red] (4.289246768012873,2.0151312862371777) circle (1.5pt);
\draw [fill=red] (-0.4685192214636319,3.215570589419363) circle (1.5pt);
\draw [fill=red] (-0.7132199683322104,0.3716239197901742) circle (1.5pt);
\end{small}
\end{tikzpicture}
\caption{On the right: Miquel's triangle theorem. The red triangle through the pivots $A_i$ closes in every position.
On the right: The closing property for four concurrent circles.}\label{fig-miquel}
\end{center}
\end{figure}
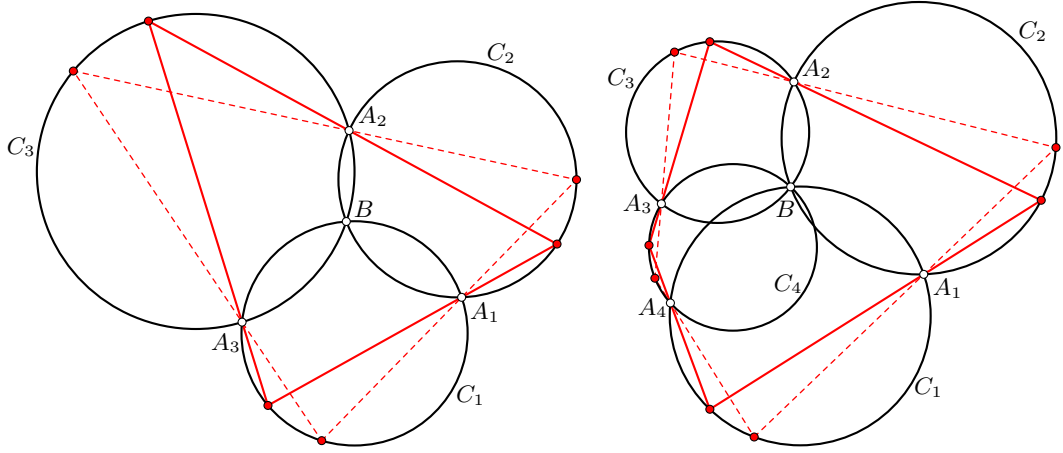

It follows from Lemma~\ref{lem-2} that the transfer angle, and hence the closing property of a chain, is
invariant unter M\"obius transformations. This means also that we can define the following variant of the 
map $\varphi_A$.
\begin{definition}
Let  $C_1\mathbin{\intersect} C_2$ be two intersecting or touching circles, $A$ 
a common point of $C_1$ and $C_2$, and $I$ a point not on $C_1\cup C_2$.
Then $\varphi_A^I: C_1\to C_2, X\mapsto \varphi_A^I(X)$, is defined as follows: 
If $X\neq A$, then the points $X,\varphi_A^I(X),A$, and $I$ are concyclic. 
If $X=A$, then the  circle through the points $A,\varphi_A^I(X)$, and $I$ is the tangent to $C_1$ in $A$ (see Figure~\ref{fig-varphi-I}).
\end{definition}
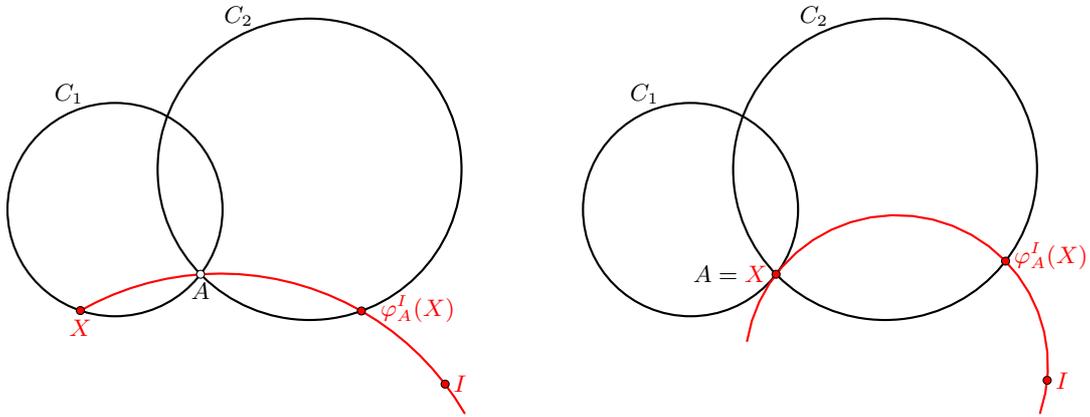
\begin{figure}[h!]
\begin{center}
\begin{tikzpicture}[line cap=round,line join=round,x=20,y=20]
\draw [line width=0.8pt] (-1.14,1.32) circle (2.012063617284503);
\draw [line width=0.8pt] (2.5,2.08) circle (2.842885857715712);
\draw [shift={(0.8339949814558981,-5.1620094624457495)},line width=0.8pt,red]  plot[domain=0.5245550849662706:2.091064294480428,variable=\t]({1.*5.275283483285311*cos(\t r)+0.*5.275283483285311*sin(\t r)},{0.*5.275283483285311*cos(\t r)+1.*5.275283483285311*sin(\t r)});
\begin{small}
\draw [fill=white] (0.46,0.1) circle (1.5pt) node[below] {$A$};
\draw[color=black] (-2,3.5) node {$C_1$};
\draw[color=black] (1.17,4.97) node {$C_2$};
\draw [fill=red] (-1.7884152627299708,-0.5847198342692874) circle (1.5pt) node[red,below] {$X$};
\draw [fill=red] (5.035994041425538,-1.9726804218735685) circle (1.5pt) node[red,right] {$I$};
\draw [fill=red] (3.4697321509154286,-0.5923808776970646) circle (1.5pt) node[red,right,xshift=4pt,yshift=1pt] {$\varphi_A^I(X)$};
\end{small}
\end{tikzpicture}\qquad\qquad
\begin{tikzpicture}[line cap=round,line join=round,x=20,y=20]
\draw [line width=0.8pt] (-1.14,1.32) circle (2.012063617284503);
\draw [line width=0.8pt] (2.5,2.08) circle (2.842885857715712);
\draw [shift={(2.713358076462534,-1.6184775199523205)},line width=.8pt,color=red]  plot[domain=-0.32375084410390365:2.9791107372620984,variable=\t]({1.*2.8338644299508076*cos(\t r)+0.*2.8338644299508076*sin(\t r)},{0.*2.8338644299508076*cos(\t r)+1.*2.8338644299508076*sin(\t r)});
\begin{small}
\draw [fill=red] (0.46,0.1) circle (1.5pt)  node[left] {$A=\textcolor{red}X$};
\draw[color=black] (-2,3.5) node {$C_1$};
\draw[color=black] (1.17,4.97) node {$C_2$};
\draw [fill=red] (5.533125366918801,-1.9007894662339724) circle (1.5pt) node[red,right] {$I$};
\draw [fill=red] (4.754154227181938,0.3477215235194438) circle (1.5pt) node[red,right,yshift=2pt] {$\varphi_A^I(X)$};
\end{small}
\end{tikzpicture}
\caption{The map $\varphi_A^I: C_1\to C_2, X\mapsto \varphi_A^I(X)$.}\label{fig-varphi-I}
\end{center}
\end{figure}
The previous results nicely carry over to the map $\varphi_A^I$. As an example we get the six circles theorem of Miquel~\cite[Th\'eor\`eme~I,  Planche~III, Fig.~1]{miquel6}:
\begin{corollary}\label{cor-6}
Let $C_1\mathbin{\intersect} C_2\mathbin{\intersect} C_3\mathbin{\intersect} C_1$ be a closed chain of three circles which all
pass through a point $B$, and $I$ a point not on any of the three circles. Let $A_i$ be the other common point of $C_i$ and $C_{i+1}$. Then, the 
map $\varphi_{A_3}^I\circ \varphi_{A_2}^I\circ \varphi_{A_1}^I$ is the identity map on $C_1$ (see Figure~\ref{fig-6}). 
\end{corollary}
\begin{figure}[h!]
\begin{center}
\begin{tikzpicture}[line cap=round,line join=round,x=20,y=20]
\draw [line width=0.8pt] (11.781798581457549,-0.8835255097257012) circle (1.5142238971634117);
\draw [line width=0.8pt] (12.085770444767505,-4.348804751459263) circle (3.0990562148536185);
\draw [line width=0.8pt,color=red] (4.963456509761086,-3.492033353329521) circle (4.566923848224915);
\draw [line width=0.8pt] (9.867408429997806,-2.0365104873094126) circle (3.369271932535286);
\draw [line width=0.8pt,color=red] (9.609137169013284,1.874307850746625) circle (2.6422827562921856);
\draw [line width=0.8pt,color=red] (9.23640398434803,-1.4331739547509639) circle (1.2873110953959725);
\begin{small}
\draw [fill=white] (10.498414378608185,-1.687141710792761) circle (1.5pt) node[left,yshift=3.5pt,xshift=2pt] {$A_1$};
\draw [fill=white] (11.936236557787378,0.6228021180524973) circle (1.5pt) node[below,xshift=1pt] {$A_{3}$};
\draw [fill=white] (13.185491893795533,-1.45143315682896) circle (1.5pt) node[right,yshift=3.5pt,xshift=-1.3pt] {$B$};
\draw[color=black] (13.28,0.2) node {$C_1$};
\draw [fill=red] (9.461296741167457,-2.70068849283711) circle (1.5pt) node[red,below,xshift=7.8pt,yshift=3pt]  {$X_2$};
\draw[color=black] (15.4,-3.5) node {$C_2$};
\draw [fill=white] (9.145308034096317,-5.327492768271888) circle (1.5pt) node[above,xshift=7.2pt,yshift=-3pt] {$A_2$};
\draw[color=black] (9.8,1.6) node {$C_3$};
\draw [fill=red] (7.466829646502241,0.3276415508643184) circle (1.5pt) node[red,left,xshift=0pt,yshift=0pt] {$X_3$};
\draw [fill=red] (10.281168097742743,-0.6810849947920117) circle (1.5pt) node[red,below,xshift=9.6pt,yshift=4pt]  {$X_1$};
\draw [fill=red] (8.385229895163278,-0.46742423239526687) circle (1.5pt) node[red,below,xshift=-1pt,yshift=0pt]  {$I$};
\end{small}
\end{tikzpicture}
\caption{Miquel's six circles theorem follows Theorem~\ref{thm-2} applied to $\varphi_{A_3}^I\circ\varphi_{A_2}^I\circ\varphi_{A_1}^I$.}\label{fig-6}
\end{center}
\end{figure}
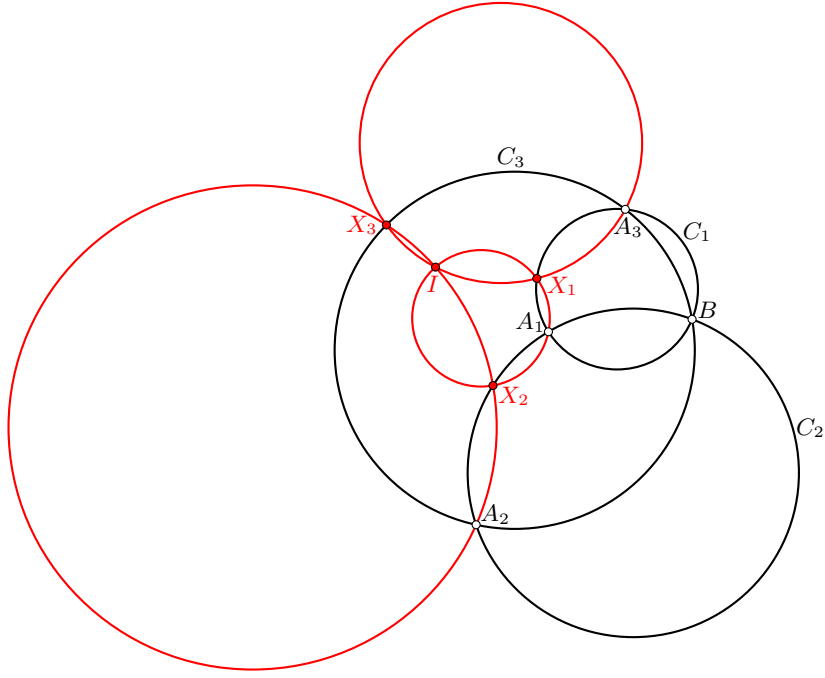

We would also like to point out that the previous closing results can be carried over to circle chains that are {\em not\/} closed.
Figure~\ref{fig-open} shows such a situation for an open chain of four circles.
The red polygon closes in every position if it closes in one position. This can be seen as follows.
Given an open chain of circles $C_1\mathbin{\intersect} C_2\mathbin{\intersect} C_3\mathbin{\intersect}\ldots  \mathbin{\intersect}C_n$
we can applying Theorem~\ref{thm-closingcondition}
to the {\em closed\/} chain 
$$
C_1\mathbin{\intersect} C_2\mathbin{\intersect} C_3\mathbin{\intersect} \ldots C_{n-1}\mathbin{\intersect} C_n \mathbin{\intersect} C_{n-1}
\mathbin{\intersect} C_{n-2}\ldots \mathbin{\intersect} C_2\mathbin{\intersect} C_1.
$$
\begin{figure}[h!]
\begin{center}
\definecolor{red}{rgb}{1.,0.,0.}
\begin{tikzpicture}[line cap=round,line join=round,x=22,y=22]
\draw[line width=.8pt,color=red] (19.51181409493095,-2.4957491784679817) -- (14.101887081770611,-0.8209258292251747) -- (12.00535646525744,-6.527790026858141) -- (2.814473192994684,-3.9640777974907726) -- (10.911377440238992,-1.7914191679535207) -- (15.223009920222651,-5.6748155472725745) -- cycle;
\draw[line width=.5pt,dash pattern=on 2pt off 2pt,color=red] (19.55614454693857,-3.0583340327250244) -- (14.83789192612888,-0.7873952579008661) -- (11.336754208605427,-7.189751835409789) -- (2.7667692799832966,-3.3586813322258426) -- (11.67540662920314,-2.3404927622848524) -- (14.49083685985578,-5.756974173145009) -- cycle;
\draw [line width=0.8pt] (14.582064777327933,-3.266437246963563) circle (2.492207197164722);
\draw [line width=0.8pt] (9.456486842944459,-4.62198706617083) circle (3.1825809143703525);
\draw [line width=0.8pt] (17.6518727139722,-2.9253474762253115) circle (1.908909803536759);
\draw [line width=0.8pt] (4.8159524882725115,-3.5017879065025848) circle (2.0541741193812424);
\begin{scriptsize}
\draw [fill=white] (16.36,-1.52) circle (1.5pt);
\draw [fill=white] (16.7,-4.58) circle (1.5pt);
\draw [fill=white] (12.12,-2.88) circle (1.5pt);
\draw [fill=white] (12.632933366822195,-4.8194934140674945) circle (1.5pt);
\draw [fill=white] (6.78,-2.9) circle (1.5pt);
\draw [fill=white] (19.51181409493095,-2.4957491784679817) circle (1.5pt);
\draw [fill=red] (14.101887081770611,-0.8209258292251747) circle (1.5pt);
\draw [fill=red] (15.223009920222651,-5.6748155472725745) circle (1.5pt);
\draw [fill=red] (6.289169709683971,-4.9333123845582) circle (1.5pt);
\draw [fill=red] (10.911377440238992,-1.7914191679535207) circle (1.5pt);
\draw [fill=red] (12.00535646525744,-6.527790026858141) circle (1.5pt);
\draw [fill=red] (2.814473192994684,-3.9640777974907726) circle (1.5pt);
\draw [fill=red] (19.55614454693857,-3.0583340327250244) circle (1.5pt);
\draw [fill=red] (14.83789192612888,-0.7873952579008661) circle (1.5pt);
\draw [fill=red] (11.336754208605427,-7.189751835409789) circle (1.5pt);
\draw [fill=red] (2.7667692799832966,-3.3586813322258426) circle (1.5pt);
\draw [fill=red] (11.67540662920314,-2.3404927622848524) circle (1.5pt);
\draw [fill=red] (14.49083685985578,-5.756974173145009) circle (1.5pt);
\end{scriptsize}
\end{tikzpicture}
\caption{A closing configuration for an open chain of circles. The red
polygon closes in every position it it closes in one position.}\label{fig-open}
\end{center}
\end{figure}
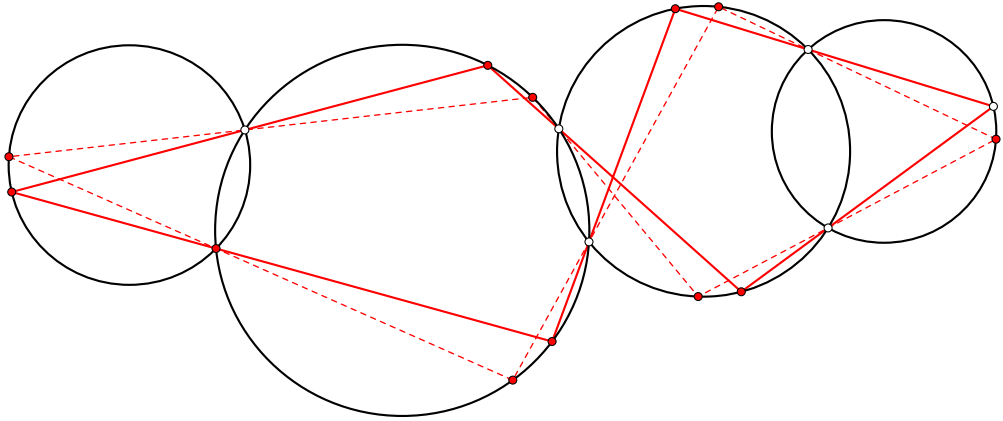

\section{Even more incidences}\label{sec-4}

Let us consider a chain $C_1\mathbin{\intersect} C_2 \mathbin{\intersect} C_3\mathbin{\intersect} \cdots $ of intersecting circles, not necessarily closed,
with centers $M_i$, and let $A_i$ be a common point of $C_i$ and $C_{i+1}$.
Take two starting points $X_1$ and $X_1'$ on $C_1$, and consider the resulting polygons $X_{i+1}=\varphi_{A_i} (X_i)$, and $X_{i+1}'=\varphi_{A_i} (X_i')$. 
By iterating the argument in the proof of Theorem~\ref{thm-1} we have
that $\sphericalangle X_1M_1X_1' = \sphericalangle X_iM_iX_i'$ for all $i$. An immediate consequence is Lemma~\ref{lem-lighthouse} below.

Let us first fix some notation that we will use throughout this section. In a situation like the one described above, we set:
\begin{itemize}
\item $\ell_j$ is the line through the points $X_j,A_j,X_{j+1}$, and $\ell_j'$ is the line through the points $X_j',A_j,X_{j+1}'$.  
\item $X_{ij}$ is the intersection of the lines $\ell_i$ and $\ell_j$, and $X_{ij}'$ is the intersection of the lines $\ell_i'$ and $\ell_j'$
\item For the circle $C_i$, we will also use the notation $C_{i-1,i}$. Similarly we will write  $X_{i-1,i}$ for the point $X_i$.
The reason for this notation will become clear below.
\end{itemize}
With these conventions we have the following.
%
\begin{lemma}\label{lem-lighthouse}
The intersection $X_{jk}$ of the lines $\ell_j$ and
$\ell_k$ lies on a fixed circle $C_{jk}$ through the points $A_j$ and $A_k$, no matter of the position of the initial point $X_1$.
The circles $C_{ij}, C_{jk},C_{ki}$ are concurrent in a point $P_{ijk}$ for all triples $i,j,k$.
\end{lemma}
\begin{proof}
Let $X_1$ and $X_1'$ be two initial points on $C_1$ and consider the resulting polygons.
Consider the points $X_{jk}$  and $X_{jk}'$ in the notation introduced above.
We have to treat  two cases.

1.~case: $X_{jk}$ and $X_{jk}'$ lie on the same side of the line $A_jA_k$  (see Figure~\ref{fig-lighthouse} on the left).
Then we have
$\sphericalangle A_kX_{jk}A_j = \sphericalangle A_kX_{jk}'A_j$ (these angles are denoted by  $\sigma$ in Figure~\ref{fig-lighthouse})
and hence the points $A_k,X_{jk},X_{jk}',A_j$ are concyclic.

2.~case:  $X_{jk}$ and $X_{jk}$ lie on the opposite sides of the line $A_jA_k$  (see Figure~\ref{fig-lighthouse} on the right).
In this case the angles $\sigma=\sphericalangle A_kX_{jk}A_j$ and $\bar\sigma= \sphericalangle A_jX_{jk}'A_k$
are supplementary, and hence  the points $A_k,X_{jk},X_{jk}',A_j$ are again concyclic. This proves the first part of the theorem.

To show that the circles $C_{ij}, C_{jk},C_{ki}$ are concurrent, consider a situation where the lines
$\ell_i$ and $\ell_j$ meet in the intersection of $C_{ij}$ and $C_{jk}$ different from the common point $A_j$. In this case, the line $\ell_k$ passes
also through this point, i.e., $X_{ij}=X_{jk}=X_{ki}$. But his means that the circle $C_{ki}$ also
passes through that same point.
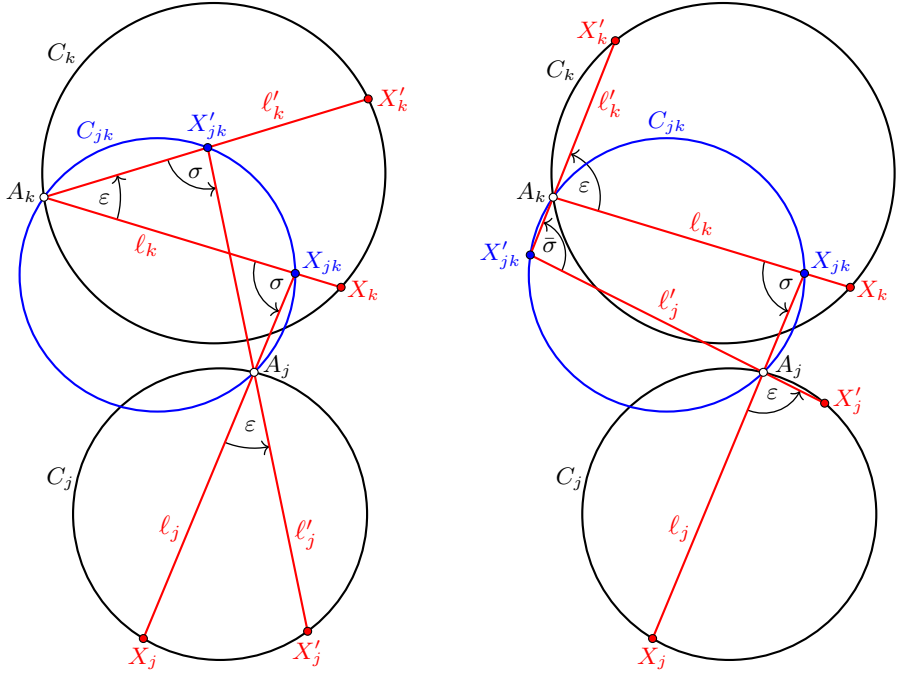
\begin{figure}
\begin{center}
\begin{tikzpicture}[line cap=round,line join=round,x=22,y=22]
\draw [shift={(2.441659895641564,-3.1759239727171233)},line width=.5pt,->]  (-112.42060651340253:1.3) arc (-112.42060651340253:-78.42060651340253:1.3);
\draw [shift={(-1.1327775679873493,-0.17349963388099823)},line width=.5pt,->]  (-16.99576407305493:1.3) arc (-16.99576407305493:17.004235926945068:1.3);
\draw [shift={(3.141421038690468,-1.4799077627000705)},line width=.5pt,->] (163.00423592694506:0.7088124356459637) arc (163.00423592694506:247.57939348659747:0.7088124356459637) ;
\draw [shift={(1.6519842095280375,0.6781126108629056)},line width=.5pt,->](-162.99576407305491:0.7088124356459637) arc (-162.99576407305491:-78.42060651340253:0.7088124356459637) ;
\draw [line width=0.8pt] (1.8624756636549666,-5.608497747060815) circle (2.5005738425822615);
\draw [line width=0.8pt] (1.7561537983080722,0.2392048470183444) circle (2.9182613707000145);
\draw [line width=0.8pt,color=blue] (0.7970042365593774,-1.5049880299093812) circle (2.34455095101767);
\draw [line width=0.8pt,color=red] (0.5578383495685777,-7.741756057931786)-- (3.141421038690468,-1.4799077627000705) node[pos=.3,left,xshift=1pt]{$\ell_j$};
\draw [line width=0.8pt,color=red] (1.6519842095280375,0.6781126108629056)-- (3.35167258724176,-7.617269029458651)node[pos=.8,right,xshift=-1pt]{$\ell_j'$};
\draw [line width=0.8pt,color=red] (-1.1327775679873493,-0.17349963388099823)-- (4.381779558075161,1.5129154881639009) node[pos=.7,above,xshift=1pt]{$\ell_k'$};
\draw [line width=0.8pt,color=red] (3.9206944616195587,-1.7180925612952203)-- (-1.1327775679873493,-0.17349963388099823) node[pos=.7,below,xshift=5pt]{$\ell_k$};
\begin{small}
\draw[color=black] (-0.85,-5) node {$C_j$};
\draw[color=black] (-0.8221514363541217,2.3) node {$C_k$};
\draw [fill=red] (0.5578383495685777,-7.741756057931786) circle (1.5pt) node[below,red] {$X_j$};
\draw [fill=red] (3.35167258724176,-7.617269029458651) circle (1.5pt) node[below,red] {$X_j'$};
\draw [fill=red] (3.9206944616195587,-1.7180925612952203) circle (1.5pt) node[right,red,xshift=-1pt,yshift=-1pt] {$X_k$};
\draw [fill=red] (4.381779558075161,1.5129154881639009) circle (1.5pt) node[right,red] {$X_k'$};
\draw [fill=white] (2.441659895641564,-3.1759239727171233) circle (1.5pt) node[right,yshift=2pt] {$A_j$};
\draw [fill=white] (-1.1327775679873493,-0.17349963388099823) circle (1.5pt)  node[left,yshift=2pt,xshift=.5pt] {$A_k$};
\draw [fill=blue] (3.141421038690468,-1.4799077627000705) circle (1.5pt) node[right,yshift=4pt,xshift=-.8pt,blue] {$X_{jk}$};
\draw[color=blue] (-0.25510148783735076,.94) node {$C_{jk}$};
\draw [fill=blue] (1.6519842095280375,0.6781126108629056) circle (1.5pt)  node[above,blue] {$X_{jk}'$};
\draw[color=black] (2.38,-4.1) node {$\varepsilon$};
\draw[color=black] (-.1,-0.2) node {$\varepsilon$};
\draw[color=black] (2.83,-1.66) node {$\sigma$};
\draw[color=black] (1.45,0.25) node {$\sigma$};
\end{small}
\end{tikzpicture}\qquad
\begin{tikzpicture}[line cap=round,line join=round,x=22,y=22]
\draw [shift={(2.441659895641564,-3.1759239727171233)},line width=.5pt,->] (-112.42060651340253:.7) arc (-112.42060651340253:-26.920606513402536:.7);
\draw [shift={(-1.1327775679873493,-0.17349963388099823)},line width=.5pt,->](-16.995764073054932:.8) arc (-16.995764073054932:68.50423592694507:.8);
\draw [shift={(3.1414210386904675,-1.4799077627000703)},line width=.5pt,->] (163.00423592694506:0.7088124356459637) arc (163.00423592694506:247.57939348659747:0.7088124356459637);
\draw [shift={(-1.5224723225847925,-1.1630119291383214)},line width=.5pt,->] (-26.920606513402536:0.6) arc (-26.920606513402536:68.50423592694507:0.6);
\draw [line width=0.8pt] (1.8624756636549666,-5.608497747060815) circle (2.5005738425822615);
\draw [line width=0.8pt] (1.7561537983080722,0.2392048470183444) circle (2.9182613707000145);
\draw [line width=0.8pt,color=blue] (0.7970042365593774,-1.504988029909382) circle (2.3445509510176694);
\draw [line width=0.8pt,color=red] (0.5578383495685777,-7.741756057931786)-- (3.1414210386904675,-1.4799077627000703)node[pos=.3,left,xshift=1pt]{$\ell_j$};
\draw [line width=0.8pt,color=red] (-1.5224723225847925,-1.1630119291383214)-- (3.484765850137177,-3.705593626236507)node[pos=.5,above,xshift=-3pt]{$\ell_j'$};
\draw [line width=0.8pt,color=red] (3.9206944616195587,-1.7180925612952205)-- (-1.1327775679873493,-0.17349963388099823)node[pos=.5,above,yshift=-1pt]{$\ell_k$};

\draw [line width=0.8pt,color=red] (-1.1327775679873493,-0.17349963388099823)-- (-0.07554900459458547,2.5110134370096917)node[pos=.6,right,xshift=-1pt]{$\ell_k'$};

\draw [line width=0.8pt,color=red] (-1.1327775679873493,-0.17349963388099823)-- (-1.5224723225847925,-1.1630119291383214);
\begin{small}
\draw[color=black] (-0.85,-5) node {$C_j$};
\draw[color=black] (-1,2.) node {$C_k$};
\draw [fill=red] (0.5578383495685777,-7.741756057931786) circle (1.5pt) node[red,below] {$X_j$};
\draw [fill=red] (3.484765850137177,-3.705593626236507) circle (1.5pt) node[red,right,yshift=1pt,xshift=-.2pt] {$X_j'$};
\draw [fill=red] (3.9206944616195587,-1.7180925612952205) circle (1.5pt) node[right,red,xshift=-1pt,yshift=-1pt] {$X_k$};
\draw [fill=red] (-0.07554900459458547,2.5110134370096917) circle (1.5pt)  node[left,red,yshift=3.5pt,xshift=1pt] {$X_k'$};
\draw [fill=white] (2.441659895641564,-3.1759239727171233) circle (1.5pt) node[right,yshift=2pt,xshift=1pt] {$A_j$};
\draw [fill=white] (-1.1327775679873493,-0.17349963388099823) circle (1.5pt) node[left,yshift=2pt,xshift=1pt] {$A_k$};
\draw [fill=blue] (3.1414210386904675,-1.4799077627000703) circle (1.5pt) node[blue,right,yshift=4pt,xshift=-1pt] {$X_{jk}$};
\draw[color=blue] (0.8081171656315947,1.1) node {$C_{jk}$};
\draw [fill=blue] (-1.5224723225847925,-1.1630119291383214) circle (1.5pt) node[blue,left] {$X_{jk}'$};
\draw[color=black] (2.6,-3.55) node {$\epsilon$};
\draw[color=black] (-.6,0) node {$\varepsilon$};
\draw[color=black] (2.83,-1.66) node {$\sigma$};
\draw[color=black] (-1.19,-1) node {$\bar\sigma$};
\end{small}
\end{tikzpicture}
\caption{Proof of Lemma~\ref{lem-lighthouse}.}\label{fig-lighthouse}
\end{center}
\end{figure}
\end{proof}
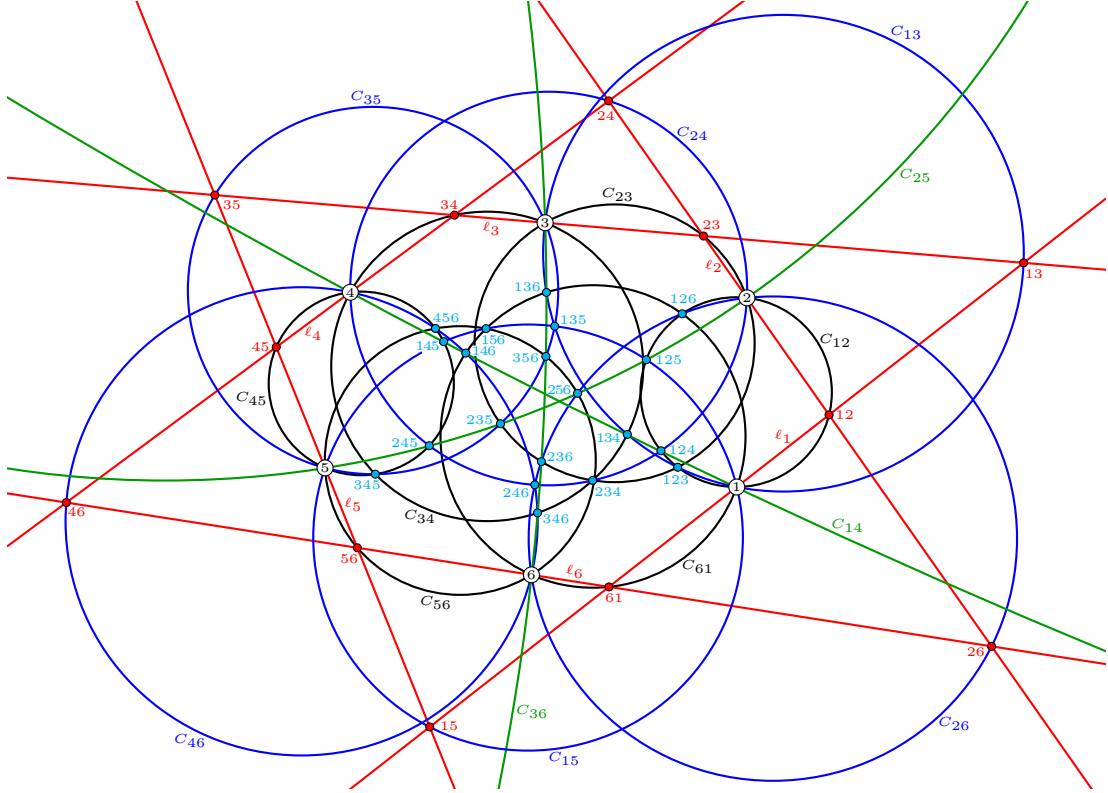
\begin{figure}[h]
\begin{center}
\begin{tikzpicture}[line cap=round,line join=round,x=20,y=20]
\clip(-5.721384969695936,-10.85148479995618) rectangle (14.823636225390633,3.997313413409888);
\draw [line width=0.8pt] (5.2343836979732234,-4.216765581031073) circle (2.8533498052877553);
\draw [line width=0.8pt] (7.913073410720727,-3.380525028916483) circle (1.7899157239613572);
\draw [line width=0.8pt] (5.638569493687873,-2.459318422200651) circle (2.619590905567516);
\draw [line width=0.8pt] (3.2552543110113064,-2.8950912080749482) circle (2.9186202046961514);
\draw [line width=0.8pt] (0.8987816300464728,-3.2160810873825163) circle (1.7319492459432346);
\draw [line width=0.8pt] (2.7515590257157183,-4.670668806167989) circle (2.5341879220389436);
\draw [line width=0.8pt,color=red,domain=-5.721384969695936:14.823636225390633] plot(\x,{(-20.156750503569697-0.5104259373002638*\x)/3.257611550610114});
\draw [line width=0.8pt,color=red,domain=-5.721384969695936:14.823636225390633] plot(\x,{(-27.291154880382855--1.8840729112975536*\x)/2.3909934548747076});
\draw [line width=0.8pt,color=red,domain=-5.721384969695936:14.823636225390633] plot(\x,{(-15.449204263691401--2.2074928382393617*\x)/-1.5369809229560634});
\draw [line width=0.8pt,color=red,domain=-5.721384969695936:14.823636225390633] plot(\x,{(-0.5380251939993448--0.25035176579578267*\x)/-2.9629857357704594});
\draw [line width=0.8pt,color=red,domain=-5.721384969695936:14.823636225390633] plot(\x,{(--3.9242849284810104-1.4546797133473137*\x)/-1.9482694431916945});
\draw [line width=0.8pt,color=red,domain=-5.721384969695936:14.823636225390633] plot(\x,{(-3.887490856162111-2.2791928235632306*\x)/0.9130287396055088});
\draw [line width=0.8pt,color=blue] (8.5995010386411,-6.144401111407776) circle (4.568378945620915);
\draw [line width=0.8pt,color=blue] (8.797967868268577,-0.7600127485953754) circle (4.496223548469916);
\draw [line width=0.8pt,color=blue] (4.405512322794316,-1.4260943572096305) circle (3.7125211131550135);
\draw [line width=0.8pt,color=blue] (1.1188778098495122,-1.463848390156992) circle (3.464457075921461);
\draw [line width=0.8pt,color=blue] (-0.21742725375315797,-5.818048197406327) circle (4.416717045249353);

\draw [line width=.8pt,color=darkgreen] (74.89907586583134,135.605) circle (155.8701990232868);
\draw [line width=.8pt,color=darkgreen] (-2.770866288532993,13.864229000653532) circle (18.912398736231296);
\draw [line width=.8pt,color=darkgreen] (-42.40117763953322,-1.7081179754974127) circle (46.76465179064406);
\draw [line width=0.8pt,color=blue] (4.019365300101478,-6.124184617168178) circle (4.019097913285547);
\begin{tiny}
\draw[color=darkgreen] (11.27,.7) node {$C_{25}$};
\draw[color=darkgreen] (4.1,-9.4) node {$C_{36}$};
\draw[color=darkgreen] (10,-5.9) node {$C_{14}$};

\draw [fill=white] (4.08244,-6.8272) circle (3pt) node[xshift=0pt,yshift=0pt] {$6$};
\draw [fill=white] (7.923652820066051,-5.17040948747296) circle (3pt) node {$1$};
\draw[color=black] (7.2,-6.67) node {$C_{61}$};
\draw [fill=white] (8.113868117764344,-1.601907648896996) circle (3pt)  node {$2$};
\draw[color=black] (9.77,-2.4) node {$C_{12}$};
\draw [fill=white] (4.338663863702232,-0.18500492893203996) circle (3pt) node {$3$};
\draw[color=black] (5.7,.34) node {$C_{23}$};
\draw [fill=white] (0.6936554964100641,-1.496321958951236) circle (3pt) node {$4$};
\draw[color=black] (2,-5.73) node {$C_{34}$};
\draw [fill=white] (0.22120354498194006,-4.809987022236773) circle (3pt) node {$5$};
\draw[color=black] (-1.15,-3.5) node {$C_{45}$};
\draw[color=black] (2.3,-7.35) node {$C_{56}$};
\draw [fill=red] (0.824832878578428,-6.316826386154046) circle (1.5pt) node[red,below,xshift=-3pt] {${56}$};
\draw[color=red] (4.9,-6.77) node {$\ell_6$};
\draw [fill=red] (5.532,-7.05448) circle (1.5pt) node[red,below,xshift=2pt] {${61}$};
\draw[color=red] (8.8,-4.22) node {$\ell_1$};
\draw [fill=red] (9.650849040720408,-3.809400487136358) circle (1.5pt)  node[red,right] {${12}$};
\draw[color=red] (7.5,-1) node {$\ell_2$};
\draw [fill=red] (7.301649599472691,-0.4353566947278226) circle (1.5pt) node[red,above,xshift=3pt,yshift=0pt] {${23}$};
\draw[color=red] (3.34,-.3) node {$\ell_3$};
\draw [fill=red] (2.6419249396017586,-0.041642245603922295) circle (1.5pt) node[red,above,xshift=-2pt] {${34}$};
\draw[color=red] (0,-2.27) node {$\ell_4$};
\draw [fill=red] (-0.6918251946235687,-2.5307941986735423) circle (1.5pt) node[red,left] {${45}$};
\draw[color=red] (.75,-5.5) node {$\ell_5$};
\draw [fill=red] (12.691216753439965,-8.176136502658435) circle (1.5pt) node[red,anchor=north east,yshift=2pt] {${26}$};
\draw[color=blue] (12,-9.63) node {$C_{26}$};
\draw [fill=red] (13.290532130522244,-0.9413757749721982) circle (1.5pt) node[red,right,yshift=-4pt,xshift=-2pt] {${13}$};
\draw[color=blue] (11.1,3.4) node {$C_{13}$};
\draw [fill=red] (5.5274502858670616,2.1128416725368693) circle (1.5pt) node[red,below,xshift=-1pt,yshift=-1.5pt] {${24}$};
\draw[color=blue] (7.1,1.5) node {$C_{24}$};
\draw [fill=red] (-1.8406865482116122,0.3371073675995738) circle (1.5pt) node[red,right,yshift=-3.5pt,xshift=.5pt] {${35}$};
\draw[color=blue] (1,2.18) node {$C_{35}$};
\draw [fill=red] (-4.619907313347611,-5.46370544995778) circle (1.5pt) node[red,right,yshift=-3.5pt,xshift=-2pt] {${46}$};
\draw[color=blue] (-2.3,-9.96) node {$C_{46}$};
\draw [fill=red] (2.1789669456332477,-9.697150042896066) circle (1.5pt) node[red,right,yshift=.5pt,xshift=1.4pt] {${15}$};
\draw [fill=red] (227.40652739148305,167.77934812505669) circle (1.5pt) node {${14}$};
\draw [fill=red] (-88.22283031116076,7.635800021837454) circle (1.5pt) node {${36}$};
\draw [fill=red] (-13.498891967864289,29.439475097643832) circle (1.5pt) node {${25}$};
\draw[color=darkgreen] (-2.8925827092822307,267.87725669541345) node {$C_{14}$};
\draw[color=darkgreen] (-12.06446717137445,29.755826829054023) node {$C_{25}$};
\draw[color=darkgreen] (-65.628272429993,37.86570361658815) node {$C_{36}$};
\draw [fill=cyan] (6.819573490766613,-4.797584034019824) circle (1.5pt) node[cyan,below,yshift=-.8pt,xshift=-.7] {${123}$};
\draw [fill=cyan] (6.506214695043864,-4.4871174800963) circle (1.5pt) node[cyan,right,xshift=.4pt,yshift=.4pt] {${124}$};
\draw[color=blue] (4.7,-10.3) node {$C_{15}$};
\draw [fill=cyan] (6.231087825784875,-2.7683789753049806) circle (1.5pt) node[cyan,right,xshift=.9pt] {${125}$};
\draw [fill=cyan] (6.903473509835637,-1.902517471758229) circle (1.5pt) node[cyan,above,yshift=1.5pt,xshift=.5pt] {${126}$};
\draw [fill=cyan] (4.3630098250761025,-1.4997244173811157) circle (1.5pt) node[cyan,left,xshift=.5pt,yshift=2pt] {${136}$};
\draw [fill=cyan] (4.517491385461691,-2.136074984816033) circle (1.5pt) node[cyan,right,xshift=-.5pt,yshift=2.3pt] {${135}$};

\draw [fill=white,white] (2.2,-2.55) circle (3.5pt);
\draw [fill=cyan] (2.4410295696915525,-2.4279697480504097) circle (1.5pt) node[cyan,below,xshift=-6pt,yshift=2.4pt] {${1\hspace{-.5pt}4\hspace{-.5pt}5}$};

\draw [fill=white,white] (3.04,-2.6) circle (1.8pt);
\draw [fill=cyan] (2.853716574415545,-2.643860892835695) circle (1.5pt) node[cyan,right,xshift=-.2pt,yshift=.9pt] {${1\hspace{-.5pt}4\hspace{-.5pt}6}$};
\draw [fill=cyan] (3.233339804868106,-2.1826986083867976) circle (1.5pt) node[cyan,below,xshift=2.7pt,yshift=-.0pt] {${1\hspace{-.5pt}5\hspace{-.5pt}6}$};
\draw [fill=cyan] (5.2275480582817515,-5.046463159583685) circle (1.5pt) node[cyan,below,xshift=5.8pt,yshift=0pt] {${234}$};
\draw [fill=cyan] (3.503456410635783,-3.9770630104340725) circle (1.5pt) node[cyan,left,xshift=-.2pt,yshift=1.4pt] {${235}$};
\draw [fill=cyan] (4.268189638946721,-4.691876498993202) circle (1.5pt) node[cyan,right,xshift=-.5pt,yshift=2pt] {${236}$};
\draw [fill=cyan] (2.171238654523535,-4.39102845274929) circle (1.5pt) node[cyan,left,xshift=-1.6pt,yshift=1pt] {${245}$};
\draw [fill=cyan] (4.945541404441234,-3.4023826609901975) circle (1.5pt) node[cyan,left,xshift=.3pt,yshift=1.4pt] {${2\hspace{-.5pt}5\hspace{-.5pt}6}$};
\draw [fill=cyan] (4.145286225917211,-5.129484073959121) circle (1.5pt) node[cyan,below,xshift=-7pt,yshift=1.2pt] {${246}$};
\draw [fill=cyan] (1.161112362327333,-4.928048019722136) circle (1.5pt)  node[cyan,below,xshift=-3pt,yshift=0pt] {${345}$};
\draw [fill=cyan] (4.3528163462496225,-2.706468353656932) circle (1.5pt) node[cyan,left,xshift=0pt,yshift=-.5pt] {${356}$};
\draw [fill=cyan] (2.2863180049051,-2.179552747698783) circle (1.5pt) node[cyan,above,xshift=4pt,yshift=0pt] {${456}$};
\draw [fill=cyan] (5.876789807894728,-4.178015811409375) circle (1.5pt) node[cyan,left,xshift=0pt,yshift=-1.3pt] {${1\hspace{-.5pt}3\hspace{-.5pt}4}$};
\draw [fill=cyan] (4.196382165765799,-5.657811291663913) circle (1.5pt) node[cyan,below,xshift=7pt,yshift=2pt] {${346}$};
\end{tiny}
\end{tikzpicture}
\caption{Illustration of Lemma~\ref{lem-lighthouse}.}\label{fig-ijk}
\end{center}
\end{figure}
Figure~\ref{fig-ijk} illustrates Lemma~\ref{lem-lighthouse} applied to a closed chain 
$C_1\mathbin{\intersect} C_2 \mathbin{\intersect} C_3\mathbin{\intersect} \cdots \mathbin{\intersect}C_6\mathbin{\intersect}C_1$
with intersection points $A_1,\ldots,A_6$ such that $\varphi_{A_6}\circ\ldots\varphi_{A_1}$ is the identity on $C_1$.
The pivots $A_1,\ldots,A_6$ are represented only as black numbers in small circles. 
It is now convenient to use the alternative notation $C_{i-1,i}$ for the circle $C_i$, and  $X_{i-1,i}$ for the point $X_i$,
since this is more coherent with the combinatorics of the situation. Observe that this convention is compatible 
with the general notation $X_{ij}$ for the intersection of the lines $\ell_i$ and $\ell_j$. Also note that 
$C_{i-1,i}$  is the circumcircle of the triangle $A_{i-1},A_i$ and $X_{i-1,i}$.
%
In Figure~\ref{fig-ijk} we drop the letter $X$ in $X_{ij}$ and just write red indices $ij$.
The common point $P_{ijk}$ of the circles $C_{ij},C_{jk}$ and $C_{ki}$ is 
indicated by light blue indices $ijk$ in the figure. Then we have:
\begin{itemize}
\item The polygon $X_{61}X_{12}X_{23}X_{34} X_{45}X_{56}X_{61}$ closes
 for every position of the starting point $X_{61}$
on the circle $C_{61}$.
\item The intersection $X_{ij}$ of the lines $\ell_i$ and $\ell_j$ lies on the circle $C_{ij}$
 for every position of the starting point $X_{61}$
on the circle $C_{61}$.
\item For all triples $i,j,k$ the circles $C_{ij},C_{jk}, C_{ki}$ meet in a common point.
\end{itemize}
Note that the points $X_{14},X_{25}$ and $X_{36}$ lie outside the region of the figure on the green circles $C_{14},C_{25}$ and $C_{36}$, respectively.

Now let us have a look at the special case of three touching circles.
\begin{corollary}\label{cor-three-touching}
Let  $C_1\mathbin{\touch} C_2 \mathbin{\touch} C_3\mathbin{\touch} C_1$ be a closed chain of three touching circles, the 
contact points being $A_1=A_4,A_2=A_5,A_3=A_6$ (see Figure~\ref{fig-three-touching}).
Then the polygon $X_1X_2\ldots X_6X_1$ with vertices $X_{i+1}  = \varphi_{A_i}(X_i)$ closes for any staring point $X_1$ on $C_1$.
Then, $A_1=X_{14}$, $A_2=X_{25}$, $A_3=X_{36}$. The points $X_{135}:=X_{13}=X_{35}=X_{51}$ and $X_{246}:=X_{24}=X_{46}=X_{62}$ lie
on the circumcircle $C$ of the triangle $A_1A_2A_3$. The lines $\ell_i$ and $\ell_{i+3}$ are orthogonal, 
and the midpoint of the segment $X_{135}X_{246}$ is the center of this circle $C$.
\end{corollary}
\begin{proof}
Consider, e.g., the circle $C_{26}$ through the points $A_2,A_6=A_3,X_{26}$. Observe that $X_{26}=A_1$ if $X_1=A_1$.
Hence $C_{26}$ is the circumcircle $C$ of the triangle $A_1A_2A_3$. Similarly $C_{24}=C_{46}=C_{13}=C_{35}=C_{51}=C$.
 It follows that $X_{13}=X_{35}=X_{51}$ and $X_{24}=X_{46}=X_{62}$ lie on $C$.
 
Finally, the sum of the transfer angles of $\varphi_{A_3}\circ \varphi_{A_2}\circ \varphi_{A_1}$ is $\pi$, i.e.,
the $X_iX_{i+3}$ is a diameter of the circle $C_i$. In particular, in the position where the line $X_1X_2$ passes through the centers of $C_1$ and $C_2$
we have that $\sphericalangle  A_1A_2X_2$ is a right angle. Hence $\ell_2$ and $\ell_5$ are orthogonal, independent of the starting point $X_1$.
Similarly, we have that $\ell_1$ and $\ell_4$, and   $\ell_3$ and $\ell_6$ are orthogonal.
In particular, it follows that the center $M$ of $C$ is the midpoint of the segment $X_{135}X_{246}$.
\end{proof}
\begin{figure}[h]
\begin{center}
\begin{tikzpicture}[line cap=round,line join=round,x=55,y=55]
\draw [line width=0.8pt] (6.64541961568973,1.7541746894948558) circle (1.4010536237440885);
\draw [line width=0.8pt] (9.352344893477351,1.3916400540768692) circle (1.3300406297380651);
\draw [line width=0.8pt] (8.31616702253702,3.4786054960532615) circle (1.);
\draw [line width=0.8pt,color=blue] (8.127885701741155,2.2686508331400055) circle (0.7067107909648517);
\draw [line width=0.4pt,dotted,blue] (7.424077226081841,2.332633421836285)-- (8.831694177400482,2.2046682444437193);
\draw [line width=.8pt,color=red] (6.518574180618487,0.3588749037111594)-- (9.472761114157564,2.71621848155905);
\draw [line width=.8pt,color=red] (9.472761114157564,2.71621848155905)-- (8.22563127649453,2.482712289585555) ;
\draw [line width=.4pt,color=red] (8.22563127649453,2.482712289585555) -- (7.424077226081841,2.332633421836285);

\draw [line width=.8pt,color=red] (7.424077226081841,2.332633421836285)-- (9.231928672797187,0.06706162659469903);
\draw [line width=.8pt,color=red] (9.231928672797187,0.06706162659469903)-- (8.406702768579516,4.474498702520968);
\draw [line width=.8pt,color=red] (8.406702768579516,4.474498702520968)-- (6.518574180618487,0.3588749037111594);
\draw [line width=.8pt,color=red] (6.772265050760934,3.149474475278564)-- (8.22563127649453,2.482712289585555);
\draw [line width=.4pt,color=red] (8.22563127649453,2.482712289585555)-- (8.831694177400482,2.2046682444437193);

\draw [line width=0.8pt,color=red] (6.772265050760934,3.149474475278564)-- (7.424077226081841,2.332633421836285);
\begin{small}
\draw [fill=white,white] (8.135,1.45) circle (4.7pt);
\draw [fill=white] (8.034074534798915,1.5681941199713034) circle (1.5pt) node[right,xshift=-3.pt,yshift=-6pt] {$A_{\hspace{-1pt}1}\hspace{-4pt}=\hspace{-3pt}\textcolor{red}{X_{\!1\hspace{-1pt}4}}$};
\draw[color=black] (5.6,2.86) node {$C_1$};
\draw[color=black] (10.66,2) node {$C_2$};
\draw[color=black] (7.52,4.273171845038602) node {$C_3$};

\draw [fill=white,white] (7.1,2.74) circle (6pt) ;
\draw [fill=white] (7.62032775,2.7604079) circle (1.5pt)  node[left,yshift=-.8pt] {$\textcolor{red}{X_{\hspace{-1pt}36}}\hspace{-3.5pt}=\hspace{-3.5pt}A_3$};

\draw [fill=white,white] (9.21,2.5) circle (6pt) ;
\draw [fill=white] (8.76087,2.5829) circle (1.5pt) node[below,xshift=19pt,yshift=2.6pt] {$A_2\hspace{-3.5pt}=\hspace{-3.5pt}\textcolor{red}{X_{\hspace{-1pt}25}}$};
\draw [fill=red] (6.518574180618487,0.3588749037111594) circle (1.5pt) node[red,below] {$X_1$};
\draw [fill=red] (9.472761114157564,2.71621848155905) circle (1.5pt) node[red,above] {$X_2$};
\draw [fill=red] (8.22563127649453,2.482712289585555) circle (1.5pt) node[red,above] {$X_3$};
\draw [fill=red] (6.772265050760934,3.149474475278564) circle (1.5pt)  node[above,red] {$X_4$};
\draw [fill=red] (9.231928672797187,0.06706162659469903) circle (1.5pt)  node[red,below] {$X_5$};
\draw [fill=red] (8.406702768579516,4.474498702520968) circle (1.5pt) node[red,above] {$X_6$};
\draw [fill=blue] (7.424077226081841,2.332633421836285) circle (1.5pt)  node[left,red] {$X_{246}$};
\draw [fill=blue] (8.831694177400482,2.2046682444437193) circle (1.5pt) node[right,red,yshift=-2pt,xshift=-.5pt] {$X_{135}$};
\draw[color=blue] (8.041155472308168,3.054075916141289) node {$C$};
\draw [fill=blue] (8.127885701741155,2.2686508331400055) circle (1.5pt) node[below,blue] {$M$};
\draw[color=blue] (7.4,1.2) node[red] {$\ell_1$};
\draw[color=red] (8.5,2.63) node {$\ell_2$};
\draw[color=red] (9.13,1.2) node {$\ell_5$};
\draw[color=red] (7.9,3.6) node {$\ell_6$};
\draw[color=red] (7.94,2.7) node {$\ell_3$};
\draw[color=red] (8.4,1) node {$\ell_4$};
\end{small}
\end{tikzpicture}
\caption{The special case of three touching circles in Corollary~\ref{cor-three-touching}.}\label{fig-three-touching}
\end{center}
\end{figure}
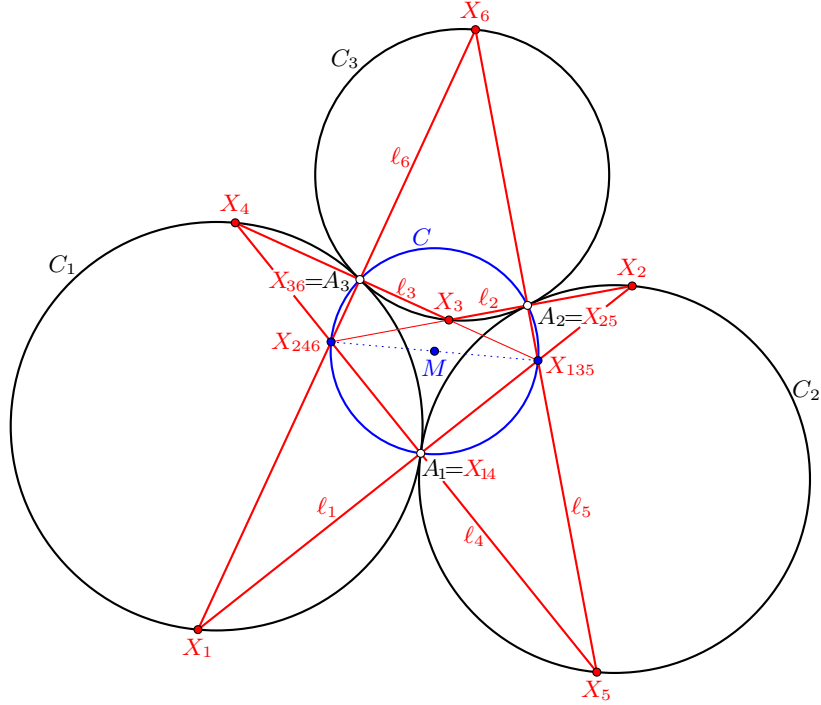

The situation of four touching circles is simpler than for three touching circles, since four is an even number (recall Theorem~\ref{thm-touching}).
\begin{corollary}\label{cor-four-touching}
Let  $C_1\mathbin{\touch} C_2 \mathbin{\touch} C_3\mathbin{\touch} C_4\mathbin{\touch} C_1$ be a closed chain of three touching circles, the 
contact points being $A_1,A_2,A_3,A_4$ (see Figure~\ref{fig-four-touching}).
Then the polygon $X_1X_2\ldots X_4X_1$ with vertices $X_{i+1}  = \varphi_{A_i}(X_i)$ closes for any staring point $X_1$ on $C_1$.
The points $X_{13}$ and $X_{24}$ lie on the circumcircle $C$ of the quadrilateral $A_1A_2A_3A_4$.
\end{corollary}
\begin{proof}
First recall that the contact points $A_1,A_2,A_3,A_4$ of four touching circles form always a cyclic quadrilateral.
Observe that $X_{13}=A_2$ if $X_2=X_4=A_2$. Hence $C_{13}=C$. Similarly $C_{24}=C$, and we are done.
\end{proof}

\begin{figure}
\begin{center}
\begin{tikzpicture}[line cap=round,line join=round,x=36,y=36]
\draw [line width=0.8pt] (6.612310217078807,1.8365365087589656) circle (1.5617033393821818);
\draw [line width=0.8pt] (9.359764170706939,2.1125964753914594) circle (1.199584826036991);
\draw [line width=0.8pt] (8.948722963503094,4.190455192784244) circle (0.9185397509095068);
\draw [line width=0.8pt] (7.037925034990063,4.362635482626819) circle (1.);
\draw [line width=0.8pt,color=blue] (7.9913969456032135,3.12647545255923) circle (1.147201785980505);
\draw [line width=.8pt,color=red] (5.455938680414409,0.7869069600943537)-- (10.248003129837585,2.9188441459869754);
\draw [line width=.8pt,color=red] (10.248003129837585,2.9188441459869754)-- (8.268585323747033,3.573099489005661);
\draw [line width=.8pt,color=red] (8.268585323747033,3.573099489005661)-- (7.778380349071575,5.034741075408523);
\draw [line width=.8pt,color=red] (7.778380349071575,5.034741075408523)-- (5.455938680414409,0.7869069600943537);
\draw [line width=0.4pt,color=red] (8.268585323747033,3.573099489005661)-- (8.716517378255594,2.237502845609563);
\draw [line width=0.4pt,color=red] (8.268585323747033,3.573099489005661)-- (7.176586370283912,3.9340370069342914);
\begin{small}
\draw [fill=white] (8.166189333323992,1.9926679032142247) circle (1.5pt) node[right,yshift=-4pt,xshift=-.6pt] {$A_1$};
\draw[color=black] (5.2,2.9) node {$C_1$};
\draw[color=black] (10.74,2.) node {$C_2$};
\draw [fill=white] (9.126973899348828,3.2893770175569115) circle (1.5pt) node[above,xshift=5pt] {$A_2$};
\draw[color=black] (9.6,5.05) node {$C_3$};
\draw[color=black] (6.15,5.15) node {$C_4$};
\draw [fill=white] (6.871779794294898,3.3765341901652604) circle (1.5pt) node[below,xshift=5pt,yshift=1pt] {$A_4$};
\draw [fill=white] (8.033889766816543,4.272889992724433) circle (1.5pt) node[above,xshift=7pt,yshift=-2pt] {$A_3$};
\draw [fill=red] (5.455938680414409,0.7869069600943537) circle (1.5pt) node[red,left,xshift=2pt,yshift=-2pt] {$X_1$};
\draw [fill=red] (10.248003129837585,2.9188441459869754) circle (1.5pt)  node[red,right,xshift=-1pt,yshift=2pt] {$X_2$};
\draw [fill=red] (8.268585323747033,3.573099489005661) circle (1.5pt)  node[red,right,xshift=-1pt,yshift=3pt] {$X_3$};
\draw [fill=red] (7.778380349071575,5.034741075408523) circle (1.5pt) node[red,above,yshift=-5pt,xshift=8pt] {$X_4$};
\draw[color=blue] (8.85,4.08) node {$C$};
\draw [fill=red] (7.176586370283912,3.9340370069342914) circle (1.5pt) node[red,left,xshift=1pt,yshift=3pt] {$X_{24}$};
\draw [fill=red] (8.716517378255594,2.237502845609563) circle (1.5pt) node[red,right,yshift=-5pt,xshift=-3pt] {$X_{13}$};
\draw[color=red] (6.77,1.2) node {$\ell_1$};
\draw[color=red] (9.55,3.01) node {$\ell_2$};
\draw[color=red] (7.8,4.617) node {$\ell_3$};
\draw[color=red] (6.05,2.2) node {$\ell_4$};
\end{small}
\end{tikzpicture}
\caption{The special case of four touching circles in Corollary~\ref{cor-four-touching}.}\label{fig-four-touching}
\end{center}
\end{figure}
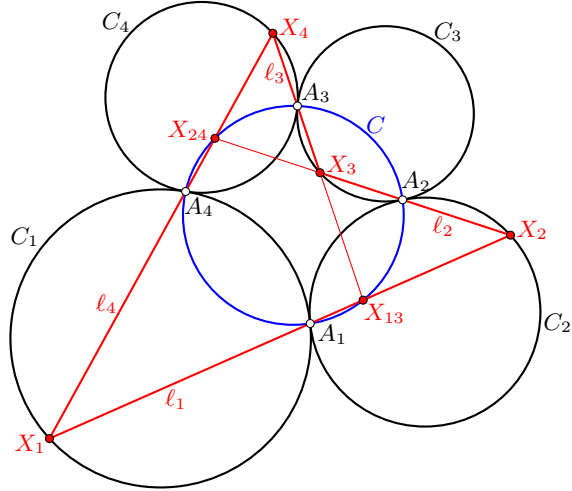

Now we can prove an extended version of Steiner's quadrilateral theorem.
\begin{corollary}\label{cor-steiner}
Consider the complete quadrilateral consisting of four lines $l_1, l_2, l_3, l_4$ such that $l_{i-1},l_{i}, l_{i+1}$ 
form a triangle with circumcircle $C_i$ for $i=1,2,3,4$ (see Figure~\ref{fig-steiner1}). 
Let $A_i$ denote the intersection of $l_{i}$ and $l_{i+1}$, $P$
the intersection of $l_2$ and $l_4$, and $Q$ the intersection of $l_1$ and $l_3$. 
Then, $\varphi:=\varphi_{A_4}\circ\varphi_{A_3}\circ\varphi_{A_2}\circ\varphi_{A_1}$ 
is the identity map on $C_1$, i.e., the resulting quadrilateral $X_1$, $X_{i+1}=\varphi_{A_i}(X_i)$ closes for any
initial point $X_1$ on $C_1$. 
The circles $C_1,C_2,C_3,C_4$ pass though a common point $S$, the Steiner point.
Moreover we have:
\begin{itemize}
\item The intersection $X_{13}$ of the lines $\ell_1$ and $\ell_3$ runs on a circle $C_{13}$ through $A_1,A_3$ and $S$.
\item The intersection $X_{24}$ of the lines $\ell_2$ and $\ell_4$ runs on a circle $C_{24}$ through $A_2,A_4$ and $S$.
\item The intersection $X$ of the lines $X_1X_3$ and $X_2X_4$ runs on a circle $C         $ through $P,Q        $ and $S$.
\item The points $X_1,X_2,X_3,X_4$ are concyclic on a circle $D$ through $S$
\item The points $X_1,X_3$ and $P$, as well as $X_2,X_4$ and $Q$ are collinear.
\end{itemize}
\end{corollary}
\begin{proof}
The fact that the polygon $X_1X_2X_3X_4$ closes for any position of the initial point $X_1$ on $C_1$ follows
directly from Corollary~\ref{cor-fivecircles}. 
Lemma~\ref{lem-lighthouse}  implies that $X_{13}$ and $X_{24}$ lie on circles $C_{13}$ and $C_{24}$ 
through $A_1,A_3$ and $A_2,A_4$, respectively.
The sum of the angles in the quadrilateral $X_3PX_1X_2$ is
$$
\sphericalangle X_3PX_1 + \sphericalangle PX_1X_2 + \sphericalangle X_1X_2X_3 + \sphericalangle X_2X_3P = 2\pi. 
$$
However, the last three of these angles do not depend on the position of $X_1$. Since $\sphericalangle X_3PX_1=0$ for $X_1=A_1$
it follows that $X_1,P,X_3$ are always collinear. Similarly, we have that $X_2,Q,X_4$ are collinear. 
Then, it follows that the quadrilateral $X_1X_3X_2X_4$ has fixed vertex and diagonal angles.
Therefore, all these quadrilaterals are similar when $X_1$ moves along $C_1$. 
If $X_1=A_1$ this is a cyclic quadrilateral (with circumcircle $C_2$), so this is always the case.
The circles $C_1$ and $C_3$ meet in a point $S$. If $X_1=S$ then $X_1=X_3$ and hence the
quadrilateral $X_1X_3X_2X_4$ with the diagonal points $X_{13}$ and $X_{24}$ degenerates to a point.  Therefore, also $C_2,C_4,C_{13}$ 
and $C_{24}$ meet in $S$.
If we consider the circles $C_1,C_3,C_2,C_4$ and the maps $\varphi_P,\varphi_{A_2},\varphi_Q,\varphi_{A_4}$ we
see that the polygon $X_1X_3X_2X_4$ is closed for every position of $X_1$ on $C_1$.
Therefore the intersection $X$ of $X_1X_3$ and $X_2X_4$ lies on a circle $C$ through $P$ and $Q$.
Also the point $X$ collapses together with the other points $X_i$ when $X_1=S$. Hence
$D$ passes also through $S$. It remains to show that the circumcircle $D$ of the points
$X_1,X_2,X_3,X_4$ passes through $S$ independent of the position of $X_1$ on $C_1$.
This follows from the fact that the angles $\sphericalangle A_4X_1S$ and $\sphericalangle SX_4A_4$
are independent of the position of $X_1$. Therefore this is also the case for $\sphericalangle X_1SX_4$.
Hence $X_1X_3X_4S$ is always a cyclic quadrilateral, since this is the case for the position $X_1=A_1$.
\end{proof}
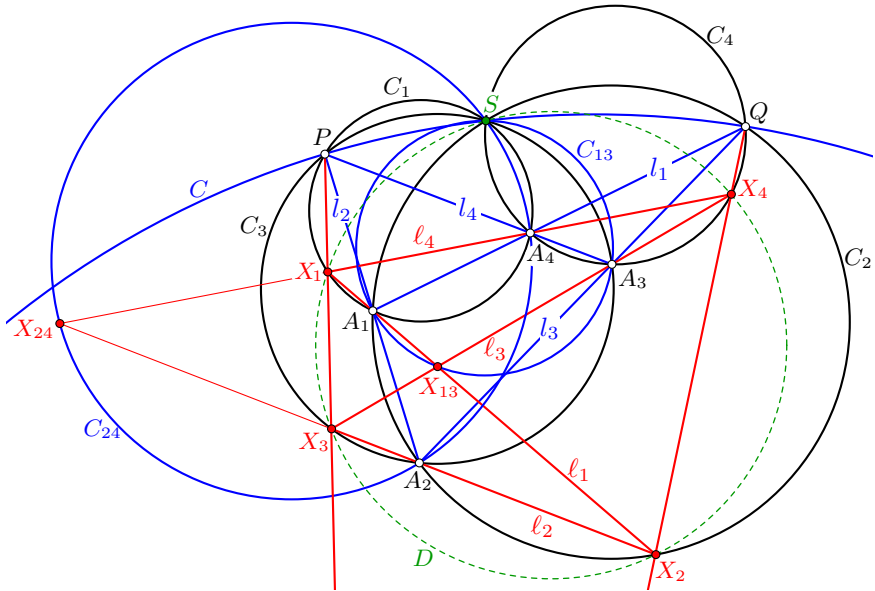
\begin{figure}[h]
\begin{center}
\begin{tikzpicture}[line cap=round,line join=round,x=12,y=12]
\clip(-9,-10.4) rectangle (18,8);
\draw [line width=0.8pt,color=blue] (9.35068740794378,-24.45598022864452) circle (28.9858322570776);
\draw [line width=0.8pt] (9.989835009681137,3.879616758843252) circle (4.065670887629065);
\draw [line width=0.8pt] (3.922395882864832,1.498239092470033) circle (3.481971839234633);
\draw [line width=0.8pt] (9.857158571976704,-2.002383045754656) circle (7.4408502443596936);
\draw [line width=0.8pt] (4.436035392822884,-0.9576187832263431) circle (5.499441458614712);
\draw [line width=0.8pt,color=blue] (3.8758420979772557,-6.428454169996841)-- (14.046709327073057,4.146919095908189) node[pos=.4,fill=white,inner sep=0pt]{$l_3$};
\draw [line width=0.8pt,color=blue] (14.046709327073057,4.146919095908189)-- (2.424883334437173,-1.6452596042888359) node[pos=.23,fill=white,inner sep=.8pt]{$l_1$};
\draw [line width=0.8pt,color=blue] (0.930830480550963,3.2799978676418196)-- (3.8758420979772557,-6.428454169996841) node[pos=.18,fill=white,inner sep=.8pt]{$l_2$};
\draw [line width=0.8pt,color=blue] (0.930830480550963,3.2799978676418196)-- (9.880875957504776,-0.18459382766224317) node[pos=.5,fill=white,inner sep=.8pt]{$l_4$};
\draw [line width=0.8pt,color=red] (1.0194609046609275,-0.42454441108498453)-- (11.251105618678402,-9.311498176191396) node[pos=.7,right]{$\ell_1$};
\draw [line width=0.8pt,color=red] (11.251105618678402,-9.311498176191396)-- (1.1374931588929638,-5.358013707288378) node[pos=.35,below,yshift=2pt]{$\ell_2$};
\draw [line width=0.8pt,color=red] (1.1374931588929638,-5.358013707288378)-- (13.604640789616619,2.018740905955818) node[pos=.41,below,yshift=1.8pt]{$\ell_3$};
\draw [line width=0.8pt,color=red] (13.604640789616619,2.018740905955818)-- (1.0194609046609275,-0.42454441108498453) node[pos=.76,above,yshift=-1.8pt]{$\ell_4$};
\draw [line width=0.8pt,color=blue] (5.909779926958069,0.3259798217377756) circle (4.003784351745917);
\draw [line width=0.8pt,color=blue] (-0.10502275827959455,-0.08246153009871746) circle (7.491255408102303);
\draw [line width=0.4pt,color=red] (1.0194609046609275,-0.42454441108498453)-- (-7.334275657245407,-2.0463378276837787);
\draw [line width=0.4pt,color=red] (-7.334275657245407,-2.0463378276837787)-- (1.1374931588929638,-5.358013707288378);

\draw [line width=.8pt,color=red,domain=0.930830480550963:18.852503924900756] plot(\x,{(--8.67933937994409-8.630640767381678*\x)/0.19685253254255441});
\draw [line width=.8pt,color=red,domain=-10.944377387727208:14.046709327073057] plot(\x,{(--28.156651244439097-2.1362390730372383*\x)/-0.4462296044012213});
\draw [line width=0.5pt,dash pattern=on 2pt off 2pt,color=darkgreen] (7.98768895727309,-2.7266027510419675) circle (7.344744260772971);

\begin{small}
\draw [fill=white] (3.8758420979772557,-6.428454169996841) circle (1.5pt) node[below,xshift=-.5pt] {$A_2$};
\draw [fill=white] (9.880875957504776,-0.18459382766224317) circle (1.5pt) node[anchor=north west,xshift=-1pt,yshift=1.3pt] {$A_3$};

\draw [fill=white,white] (7.51,.08) circle (3.9pt);
\draw [fill=white] (7.333915148232212,0.8013429802336188) circle (1.5pt) node[below,xshift=4pt,yshift=-1pt] {$A_4$};
\draw [fill=white] (2.424883334437173,-1.6452596042888359) circle (1.5pt)  node[anchor=north east,yshift=2.7pt,xshift=2.7pt] {$A_1$};
\draw [fill=white] (14.046709327073057,4.146919095908189) circle (1.5pt)node[above,xshift=5pt,yshift=-2pt]{$Q$};
\draw[color=black] (13.27,7) node {$C_4$};
\draw [fill=white] (0.930830480550963,3.2799978676418196) circle (1.5pt) node[above,xshift=-1pt]{$P$};
\draw[color=black] (3.2,5.37) node {$C_1$};
\draw[color=black] (17.57,0) node {$C_2$};
\draw[color=black] (-1.27,1) node {$C_3$};
\draw[color=blue] (-3,2.18) node {$C$};
\draw[color=darkgreen] (4,-9.4) node {$D$};

\draw [fill=white,white] (.445,-0.5) circle (4.8pt);
\draw [fill=white,white] (.3,-0.5) circle (4.8pt);
\draw [fill=red] (1.0194609046609275,-0.42454441108498453) circle (1.5pt) node[left,red,xshift=1.8pt] {$X_{\hspace{-.8pt}1}$};
\draw [fill=red] (11.251105618678402,-9.311498176191396) circle (1.5pt) node[below,red,xshift=5.7pt,yshift=.57pt] {$X_2$};
\draw [fill=red] (1.1374931588929638,-5.358013707288378) circle (1.5pt) node[red,below,xshift=-6.5pt,yshift=2pt] {$X_3$};
\draw [fill=red] (13.604640789616619,2.018740905955818) circle (1.5pt) node[right,red,yshift=2pt,xshift=-1pt] {$X_4$};

\draw [fill=white,white] (6.044,4.88) circle (4.35pt);
\draw [fill=darkgreen] (5.949139375030244,4.329570706073101) circle (1.5pt) node[darkgreen,above,xshift=2pt] {$S$};
\draw [fill=red] (4.445749117797924,-3.4005336862365834) circle (1.5pt) node[red,below,yshift=-.99pt,xshift=.6pt] {$X_{13}$};
\draw[color=blue] (9.37,3.4) node {$C_{13}$};
\draw [fill=red] (-7.334275657245407,-2.0463378276837787) circle (1.5pt) node[left,red,xshift=1.3pt,yshift=-2pt] {$X_{24}$};
\draw[color=blue] (-6,-5.355) node {$C_{24}$};
\end{small}
\end{tikzpicture}
\caption{Steiner's quadrilateral theorem follows from Corollary~\ref{cor-fivecircles} and Lemma~\ref{lem-lighthouse} when we take $n=4$.}\label{fig-steiner1}
\end{center}
\end{figure}

Figure~\ref{fig-last} illustrates Lemma~\ref{lem-lighthouse} applied to a closed chain of three intersection circles
$C_1,C_2,C_3$ with intersection points $A_i,B_i$. It follows from Theorem~\ref{thm-ab} that
$\varphi_{B_3}\circ \varphi_{B_2}\circ \varphi_{B_1}\circ \varphi_{A_3}\circ \varphi_{A_2}\circ \varphi_{A_1}$
is the identity map on $C_1$, i.e., the corresponding polygon $X_1X_2\ldots X_6X_1$ closes 
for every position of the starting point $X_1$ on $C_1$. The intersections $X_{ij}$ of the lines
$\ell_i$ and $\ell_j$ lie on the blue circles $C_{ij}$, where $C_{24}=C_{46}=C_{62}=:C_{246}$ and
 $C_{13}=C_{35}=C_{51}=:C_{135}$. Moreover $C_{25},C_{36}$ and $C_{14}$ are touching $C_{246}$ and $C_{135}$.

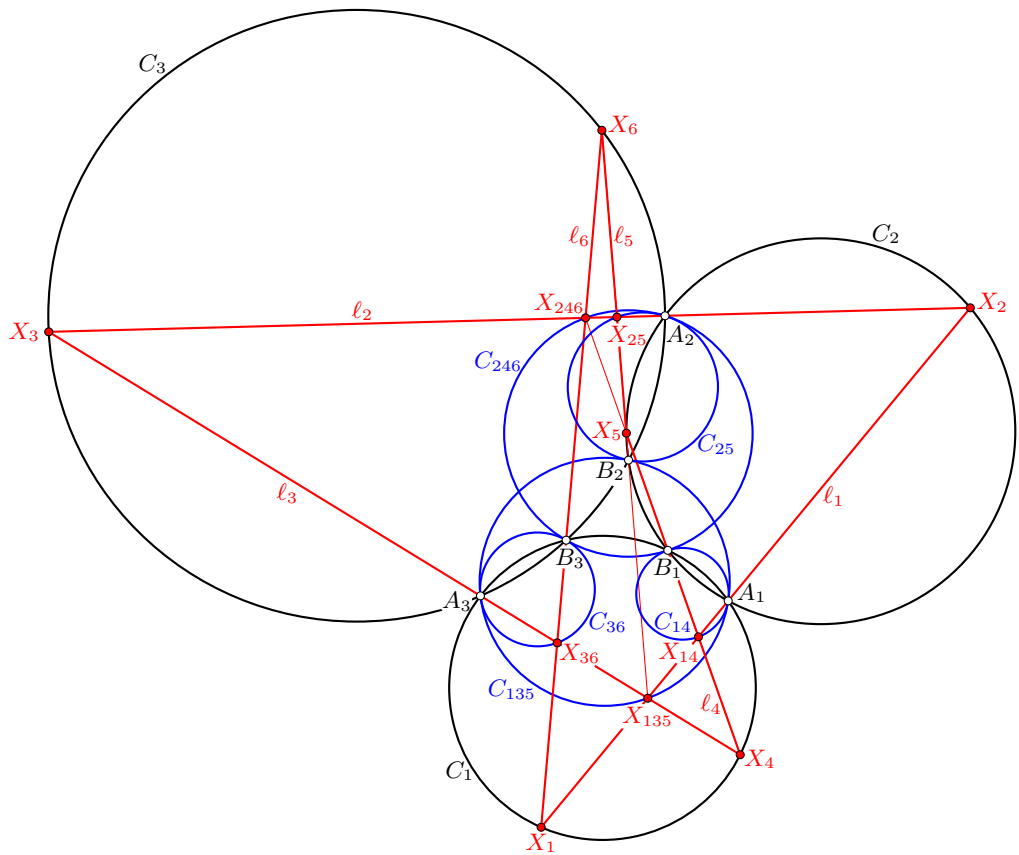
\begin{figure}
\begin{center}
\begin{tikzpicture}[line cap=round,line join=round,x=28,y=28]
\draw [line width=0.8pt] (1.2076756325551938,-0.18310705716457115) circle (2.048869230914672);
\draw [line width=0.8pt] (4.125537993195853,3.275100185076211) circle (2.5991780004543856);
\draw [line width=0.8pt] (-2.077621247573548,4.831293444084563) circle (4.121041194814751);
\draw [line width=0.8pt,color=blue] (1.553464829548858,3.24459679821055) circle (1.6602443728314136);
\draw [line width=0.8pt,color=blue] (1.236710465226126,1.2462043002339191) circle (1.6708761097233742);
\draw [line width=0.8pt,color=red] (0.38820125302843567,-2.0609587197535295)-- (6.123100185646636,4.938070889970703);
\draw [line width=0.8pt,color=red] (6.123100185646636,4.938070889970703)-- (-6.1929578433388475,4.614532670615132);
\draw [line width=0.8pt,color=red] (-6.1929578433388475,4.614532670615132)-- (3.047984777553276,-1.0837331595096615);
\draw [line width=0.4pt,color=red] (3.047984777553276,-1.0837331595096615)-- (0.9808998306590092,4.8029872177026105);
\draw [line width=0.8pt,color=red] (1.5264886485955338,3.2492393579682437)--(3.047984777553276,-1.0837331595096615);

\draw [line width=0.4pt,color=red] (1.199145072277406,7.330449904662937)-- (1.8129420089459765,-0.32216561214853684);
\draw [line width=0.8pt,color=red] (1.199145072277406,7.330449904662937)--(1.5264886485955338,3.2492393579682437) ;

\draw [line width=0.8pt,color=red] (1.199145072277406,7.330449904662937)-- (0.38820125302843567,-2.0609587197535295);

\draw [line width=0.8pt,color=blue] (2.2771987453138425,1.0844758912901236) circle (0.6178937075630226);
\draw [line width=0.8pt,color=blue] (0.33836830731081596,1.143730439261059) circle (0.7667082465084613);
\draw [line width=.8pt,color=blue] (1.7471751216085785,3.871762647634691) circle (1.003844586166484);
\begin{small}
\draw[color=black] (-0.7,-1.3) node {$C_1$};
\draw[color=black] (5,5.937401621205894) node {$C_2$};

\draw [fill=white,white] (1.4,2.68) circle (4.pt) ;
\draw [fill=white] (1.5555935547831958,2.8863690935516075) circle (1.5pt) node[left,xshift=2pt,yshift=-4pt] {$B_2$};
\draw[color=black] (-4.8,8.223054932018373) node {$C_3$};

\draw [fill=white,white] (-0.8,.96) circle (4.5pt);
\draw [fill=white,white] (-0.7,.977) circle (4.5pt);
\draw [fill=white] (-0.4233999092799136,1.0568355633692141) circle (1.5pt) node[left,xshift=0pt,yshift=-2pt] {$A_3$};

\draw [fill=white,white] (0.7222342783560433,1.6) circle (3.6pt) ;
\draw [fill=white,white] (0.7222342783560433,1.55) circle (3.4pt) ;

\draw [fill=white] (0.7222342783560433,1.807423479406565) circle (1.5pt) node[below,xshift=1.5pt,yshift=0pt] {$B_3$};

\draw [fill=white,white] (2.080902201065873,1.5) circle (2.3pt);
\draw [fill=white,white] (2.14,1.36) circle (3.3pt);
\draw [fill=white] (2.080902201065873,1.6703600953460813) circle (1.5pt)  node[below,xshift=0pt,yshift=0pt] {$B_1$};
\draw [fill=white] (2.887760858925848,0.9895731027767276) circle (1.5pt) node[right,xshift=0pt,yshift=3pt] {$A_1$};
\draw [fill=red] (0.38820125302843567,-2.0609587197535295) circle (1.5pt) node[red,below,xshift=0pt,yshift=1pt] {$X_1$};
\draw [fill=white] (2.043419928385371,4.830899221660277) circle (1.5pt) node[below,xshift=5.9pt,yshift=-.5pt] {$A_2$};
\draw [fill=red] (6.123100185646636,4.938070889970703) circle (1.5pt) node[red,right,xshift=-1pt,yshift=2pt] {$X_2$};
\draw [fill=red] (-6.1929578433388475,4.614532670615132) circle (1.5pt) node[red,left,xshift=0pt,yshift=0pt] {$X_3$};
\draw [fill=red] (3.047984777553276,-1.0837331595096615) circle (1.5pt) node[right,red,xshift=-1.8pt,yshift=-1.6pt] {$X_4$};
\draw [fill=red] (1.5264886485955338,3.2492393579682437) circle (1.5pt) node[red,left,xshift=2pt,yshift=.8pt] {$X_5$};
\draw [fill=red] (1.199145072277406,7.330449904662937) circle (1.5pt) node[red,right,xshift=-1pt,yshift=2pt] {$X_6$};
\draw [fill=red] (0.9808998306590092,4.8029872177026105) circle (1.5pt) node[red,above,xshift=-9.8pt,yshift=-1.5pt] {$X_{246}$};

\draw [fill=white,white] (1.7,-0.6) circle (5.5pt);
\draw [fill=white,white] (1.5,-0.6) circle (3.5pt);
\draw [fill=red] (1.8129420089459765,-0.32216561214853684) circle (1.5pt) node[red,below,xshift=0pt,yshift=0pt] {$X_{135}$};
\draw[color=blue] (-0.18598524276883427,4.2197804941456285) node {$C_{246}$};
\draw[color=blue] (-0.,-.2) node {$C_{135}$};
\draw[color=red] (4.3,2.4) node {$\ell_1$};
\draw[color=red] (-2,4.9) node {$\ell_2$};
\draw[color=red] (-3,2.44) node {$\ell_3$};
\draw[color=red] (2.66,-.4) node {$\ell_4$};
\draw[color=red] (1.5,5.9) node {$\ell_5$};
\draw[color=red] (0.9,5.9) node {$\ell_6$};

\draw [fill=white,white] (2.3,0.3) circle (4pt);
\draw [fill=white,white] (2.26,0.26) circle (3.4pt);
\draw [fill=red] (2.490290453844386,0.5044892038553924) circle (1.5pt) node[red,left,xshift=3.8pt,yshift=-5.8pt] {$X_{14}$};
\draw[color=blue] (2.15,.7) node {$C_{14}$};

\draw [fill=white,white] (0.86,0.25) circle (6.5pt) ;
\draw [fill=white,white] (0.91,0.25) circle (6.5pt) ;

\draw [fill=red] (0.6027813741428346,0.42405880208715574) circle (1.5pt) node[red,right,xshift=-2.5pt,yshift=-4pt] {$X_{36}$};
\draw[color=blue] (1.28,0.7) node {$C_{36}$};

\draw [fill=white,white] (1.4009813420845583,4.57) circle (4.2pt);
\draw [fill=red] (1.4009813420845583,4.814022601253984) circle (1.5pt) node[red,below,xshift=4pt,yshift=0pt] {$X_{25}$};
\draw[color=blue] (2.73,3.1) node {$C_{25}$};
\end{small}
\end{tikzpicture}
\caption{Lemma~\ref{lem-lighthouse} applied to a closed chain of three intersecting circles.}\label{fig-last}
\end{center}
\end{figure}

As a final remark we mention the following.
\begin{corollary}\label{cor-verylast}
Let $C_1\mathbin{\intersect} C_2 \mathbin{\intersect} C_3\mathbin{\intersect} \cdots \mathbin{\intersect}C_n\mathbin{\intersect}C_1$ be a closed
chain of circles. Let $A_i,B_i$ be the intersection points of $C_i$ and $C_{i+1}$, where $A_i=B_i$ is allowed. 
Let $\delta_{A_i}:=\sphericalangle A_iM_iB_i, \gamma_{A_i}:=\sphericalangle B_iM_{i+1}A_i$ and 
$\varphi:=\varphi_{A_n}\circ\cdots\circ\varphi_{A_2}\circ\varphi_{A_1}$. 
If $\sum_{i=1}^n \delta_{A_i}+\gamma_{A_i}$ is a rational multiple of $\pi$, then there is a natural number $k$ such that the map $\varphi^k$ is the identity
on $C_1$. That is, starting with any point $X$ on $C_1$, the resulting polygon will eventually close on $X$ after a finite number of steps. 
The points $X_i X_{i+n} X_{i+2n} X_{i+kn}$ form a regular polygon in the circle $C_i$ (see Figure~\ref{fig-verylast}).
\end{corollary}
Observe that Corollary~\ref{cor-verylast} can be interpreted as a stacked version of  the Lighthouse theorem~\cite{guy}:
Each pair of points $A_i,A_{i+1}$ can be considered as lighthouses while the lines $\ell_k$ passing through them correspond to
the light beams.
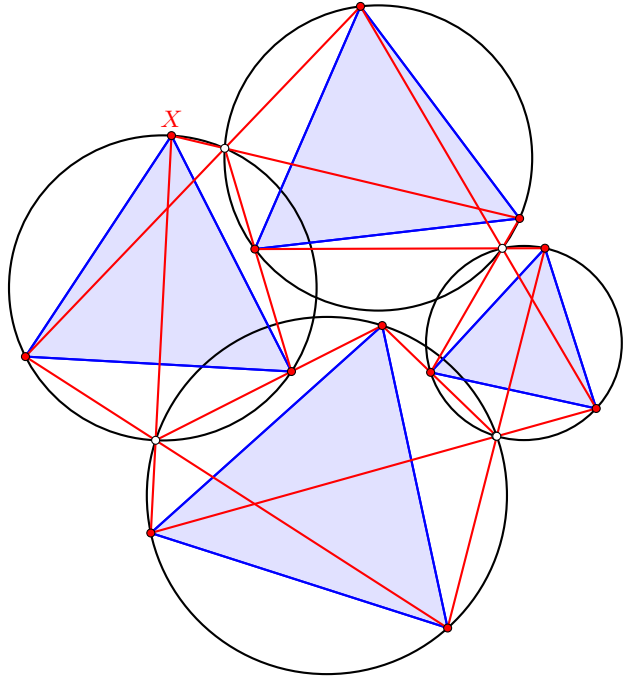
\begin{figure}
\begin{center}
\definecolor{blue}{rgb}{0.,0.,1.}
\begin{tikzpicture}[line cap=round,line join=round,x=12,y=12]
\draw[line width=.8pt,color=blue,fill=blue!12] (-3.602292108902038,-8.033988857302733) -- (5.653488048177753,-11.021902390165895) -- (3.6132069934086704,-1.5122048758593065) -- cycle;
\draw[line width=.8pt,color=blue,fill=blue!12] (-7.511228918234917,-2.488194286183146) -- (0.7827632059491996,-2.9573546661425096) -- (-2.957928048648881,4.460033402168683) -- cycle;
\draw[line width=.8pt,color=blue,fill=blue!12] (8.688396009725679,0.9198226975831079) -- (10.286379547201962,-4.119450465323652) -- (5.123249202777405,-2.983708222154042) -- cycle;
\draw[line width=.8pt,color=blue,fill=blue!12] (2.936472766822022,8.522831093639889) -- (7.897287676426038,1.8546790087165022) -- (-0.3579088802177808,0.8925633159885209) -- cycle;
\draw [line width=0.8pt] (3.4919505210100965,3.7566911394483067) circle (4.798400316599185);
\draw [line width=0.8pt] (8.032674919901682,-2.061111996631531) circle (3.0522027781280348);
\draw [line width=0.8pt] (-3.2287979203115356,-0.3285051833856544) circle (4.796193519106301);
\draw [line width=0.8pt] (1.8881343108947917,-6.856032041109314) circle (5.615368601616422);
\draw [line width=.8pt,color=blue] (-3.602292108902038,-8.033988857302733)-- (5.653488048177753,-11.021902390165895);
\draw [line width=.8pt,color=blue] (5.653488048177753,-11.021902390165895)-- (3.6132069934086704,-1.5122048758593065);
\draw [line width=.8pt,color=blue] (3.6132069934086704,-1.5122048758593065)-- (-3.602292108902038,-8.033988857302733);
\draw [line width=.8pt,color=blue] (-7.511228918234917,-2.488194286183146)-- (0.7827632059491996,-2.9573546661425096);
\draw [line width=.8pt,color=blue] (0.7827632059491996,-2.9573546661425096)-- (-2.957928048648881,4.460033402168683);
\draw [line width=.8pt,color=blue] (-2.957928048648881,4.460033402168683)-- (-7.511228918234917,-2.488194286183146);
\draw [line width=.8pt,color=blue] (8.688396009725679,0.9198226975831079)-- (10.286379547201962,-4.119450465323652);
\draw [line width=.8pt,color=blue] (10.286379547201962,-4.119450465323652)-- (5.123249202777405,-2.983708222154042);
\draw [line width=.8pt,color=blue] (5.123249202777405,-2.983708222154042)-- (8.688396009725679,0.9198226975831079);
\draw [line width=.8pt,color=blue] (2.936472766822022,8.522831093639889)-- (7.897287676426038,1.8546790087165022);
\draw [line width=.8pt,color=blue] (7.897287676426038,1.8546790087165022)-- (-0.3579088802177808,0.8925633159885209);
\draw [line width=.8pt,color=blue] (-0.3579088802177808,0.8925633159885209)-- (2.936472766822022,8.522831093639889);
\draw [line width=0.8pt,color=red] (-2.957928048648881,4.460033402168683)-- (-3.602292108902038,-8.033988857302733);
\draw [line width=0.8pt,color=red] (2.936472766822022,8.522831093639889)-- (-7.511228918234917,-2.488194286183146);
\draw [line width=0.8pt,color=red] (-7.511228918234917,-2.488194286183146)-- (5.653488048177753,-11.021902390165895);
\draw [line width=0.8pt,color=red] (-3.602292108902038,-8.033988857302733)-- (10.286379547201962,-4.119450465323652);
\draw [line width=0.8pt,color=red] (-2.957928048648881,4.460033402168683)-- (7.897287676426038,1.8546790087165022);
\draw [line width=0.8pt,color=red] (8.688396009725679,0.9198226975831079)-- (5.653488048177753,-11.021902390165895);
\draw [line width=0.8pt,color=red] (-0.3579088802177808,0.8925633159885209)-- (8.688396009725679,0.9198226975831079);
\draw [line width=0.8pt,color=red] (-1.2967690316272096,4.061339534369917)-- (0.7827632059491996,-2.9573546661425096);
\draw [line width=0.8pt,color=red] (10.286379547201962,-4.119450465323652)-- (2.936472766822022,8.522831093639889);
\draw [line width=0.8pt,color=red] (-3.4519810400220665,-5.119503155394571)-- (3.6132069934086704,-1.5122048758593065);
\draw [line width=0.8pt,color=red] (3.6132069934086704,-1.5122048758593065)-- (7.185587679434274,-4.993412289208981);
\draw [line width=0.8pt,color=red] (7.897287676426038,1.8546790087165022)-- (5.123249202777405,-2.983708222154042);
\begin{small}
\draw [fill=white] (7.359000900499327,0.9158168091001025) circle (1.5pt);
\draw [fill=white] (7.158953323455583,-4.985586713690337) circle (1.5pt);
\draw [fill=white] (-1.2967690316272096,4.061339534369917) circle (1.5pt);
\draw [fill=white] (-3.4519810400220665,-5.119503155394571) circle (1.5pt);
\draw [fill=white] (7.185587679434274,-4.993412289208981) circle (1.5pt);
\draw [fill=red] (-3.602292108902038,-8.033988857302733) circle (1.5pt);
\draw [fill=red] (10.286379547201962,-4.119450465323652) circle (1.5pt);
\draw [fill=red] (2.936472766822022,8.522831093639889) circle (1.5pt);
\draw [fill=red] (-7.511228918234917,-2.488194286183146) circle (1.5pt);
\draw [fill=red] (5.653488048177753,-11.021902390165895) circle (1.5pt);
\draw [fill=red] (8.688396009725679,0.9198226975831079) circle (1.5pt);
\draw [fill=red] (-0.3579088802177808,0.8925633159885209) circle (1.5pt);
\draw [fill=red] (0.7827632059491996,-2.9573546661425096) circle (1.5pt);
\draw [fill=red] (3.6132069934086704,-1.5122048758593065) circle (1.5pt);
\draw [fill=red] (5.123249202777405,-2.983708222154042) circle (1.5pt);
\draw [fill=red] (7.897287676426038,1.8546790087165022) circle (1.5pt);
\draw [fill=red] (-2.957928048648881,4.460033402168683) circle (1.5pt) node[above,red]{$X$};
\end{small}
\end{tikzpicture}
\caption{Illustration of Corollary~\ref{cor-verylast} with four circles such that the polygon closes after three runs.
Then the vertices on each circle form a regular triangle. The triangles ``dance'' synchronized pirouettes when $X$ runs on the circle.
If three of the circles are given, it is an easy exercise  to
construct a matching fourth circle.}\label{fig-verylast}
\end{center}
\end{figure}

\bibliographystyle{plain}

\end{document}